\documentclass[11pt,reqno,a4paper]{amsart}
\usepackage{amsfonts}
\usepackage[utf8]{inputenc}
\usepackage{graphicx}
\usepackage{color}
\usepackage{transparent}
\usepackage{ulem}
\usepackage{cancel}
\usepackage{mathtools}
\usepackage{esint} 
\usepackage[active]{srcltx}
\usepackage[colorlinks,pdfpagelabels,pdfstartview = FitH,bookmarksopen = true,bookmarksnumbered = true,linkcolor = blue,plainpages = false,hypertexnames = false,citecolor = red] {hyperref}
\usepackage{amssymb, amsmath, amsthm}
\usepackage{scalerel}
\usepackage{latexsym}

\renewcommand{\O }{\Omega }

\newcommand{\ba}{\begin{align*}}
\newcommand{\ea}{\end{align*}}

\newtheorem{innercustomthm}{Theorem}  
\newenvironment{customthm}[1]
  {\renewcommand\theinnercustomthm{#1}\innercustomthm}
  {\endinnercustomthm}
\numberwithin{equation}{section}
 \usepackage{epsfig,color,xcolor}
 
 \newtheorem{theorem}{Theorem}[section]
\newtheorem{proposition}[theorem]{Proposition}
\newtheorem{lemma}[theorem]{Lemma}
\newtheorem{corollary}[theorem]{Corollary}
\newtheorem{remark}[theorem]{Remark}



\newcommand{\ble}{\begin{lemma}}
\newcommand{\ele}{\end{lemma}}
\newcommand{\be}{\begin{equation*}}
\newcommand{\ee}{\end{equation*}}

\newcommand{\bel}{\begin{equation}}
\newcommand{\eel}{\end{equation}}

\DeclareMathOperator{\Id}{Id}
\DeclareMathOperator{\dist}{dist}

\newcommand{\ep}{\varepsilon}
\newcommand{\fr}{\frac }

\newcommand{\N}{\mathbb{N}}
\newcommand{\na}{\nabla}

\newcommand{\R}{\mathbb{R}}

\renewcommand{\to}{\rightarrow}

\newcommand{\xip}{x_{i,p}}

\newcommand{\xjp}{x_{j,p}}

\newcommand{\mip}{\mu_{i,p}}

\newcommand{\upp}{u_p}
\newcommand{\coorduno}{^{ \scaleto{1}{3.5pt} }}
\newcommand{\coorddue}{^{ \scaleto{2}{3.5pt} }}

\numberwithin{equation}{section}
\begin{document}

\title[]{Sharp boundary concentration for a\\ two-dimensional  nonlinear Neumann problem}

\author[F. De Marchis]{Francesca De Marchis}
\address{Francesca De Marchis, Dipartimento di Matematica, Sapienza Universit\`a di Roma, P.le Aldo Moro 5, 00185 Roma (Italy)}
\email{demarchis@mat.uniroma1.it}

\author[H. Fourti]{Habib Fourti}
\address{Habib Fourti, University of Monastir, Faculty of Sciences, 5019 Monastir (Tunisia)\\ Laboratory LR 13 ES 21, University of Sfax, Faculty of Sciences,
3000 Sfax (Tunisia)}
\email{habib.fourti@fsm.rnu.tn}

\author[I. Ianni]{Isabella Ianni}
\address{Isabella Ianni, Dipartimento di Scienze di Base e Applicate per l'Ingegneria, Sapienza Universit\`a di Roma, Via Antonio Scarpa 10, 00161 Roma (Italy)}
\email{isabella.ianni@uniroma1.it}



\thanks{F. De Marchis and I. Ianni are partially supported by  INDAM - GNAMPA and  {\sl Fondi Ateneo - Sapienza},  I. Ianni is also partially supported by  MUR-PRIN $2022$7HX$33$Z}

\maketitle

\begin{abstract}
We consider the elliptic equation $-\Delta u+ u=0$ in a bounded, smooth domain $\Omega\subset\mathbb R^{2}$ subject to the nonlinear Neumann boundary condition $\partial u/\partial\nu =
|u|^{p-1}u$ on $\partial\O$ and 
study the asymptotic behavior as the exponent $p\rightarrow +\infty$ of 
families of positive solutions $u_p$ satisfying uniform energy bounds. We 
prove energy quantization and 
characterize the boundary concentration. In particular we describe the local asymptotic profile of the solutions around each concentration point and get sharp convergence results for the $L^{\infty}$-norm.
 \end{abstract}

\section{Introduction} \label{SectionIntroduction} 

Let $\Omega$ be a bounded domain in $\mathbb R^{2}$ with smooth boundary $\partial\Omega$. This paper deals with the analysis of solutions of the boundary value problem\begin{equation}\label{problem1}
\left\{\begin{array}{lr}\Delta u= u&\qquad  \mbox{ in }\Omega\\
u>0 &\qquad  \mbox{ in }\Omega\\
\frac{\partial u}{\partial \nu}=u^{p}&\ \mbox{ on }\partial
\Omega
\end{array}\right.
\end{equation}
where $\nu$ denotes the outer unit normal vector to $\partial\Omega$ and $p>1$. Two dimensional elliptic equations with nonlinear Neumann boundary conditions arise in many fields (conformal geometry, corrosion modelling, etc...) see for instance \cite{BFS,  BryanVogelius,  ChangLiu, Cherrier, Cruz, DavilaDM,  GL, Aleks, KavianVogelius, LWZ,    Soriano, MedvilleVogelius,  WW} and in particular, \cite{Castro2, Castro1, F, Takahashi} where problem \eqref{problem1} is considered.

\

Observe that solutions to \eqref{problem1} correspond to critical points in $H^{1}(\Omega)$ of the free energy functional
$$E_p(u):=\displaystyle \frac{1}{2}\int_{\Omega} (|\nabla u|^2+u^2) \ dx-\frac{1}{p+1}\int_{\partial \Omega}
u^{p+1}\ d\sigma,$$ and  by the compact trace and Sobolev embeddings
$H^1(\Omega)\hookrightarrow
H^\frac{1}{2}(\partial\Omega)\hookrightarrow L^p(\partial\Omega)$, one can derive the existence of at least a 
solution for any fixed exponent $p>1$ by standard variational methods (see for instance \cite{Takahashi}). For multiplicity results for $p$ large enough see Castro  (\cite{Castro2}) and for sign-changing solutions see for instance \cite{LWZ}.

\

This paper is devoted to the study of the asymptotic
behavior, as $p\rightarrow +\infty$, of general families of non-trivial solutions $u_{p}$ to
\eqref{problem1} under a uniform bound of their energy,  namely we assume \begin{equation}\label{energylimit}
p\displaystyle\int_{\Omega}(|\nabla u_p|^2+u_p^2) \ \ dx \rightarrow
\beta \in \mathbb{R}, \quad \hbox{as }p\rightarrow +\infty.
\end{equation}

\

In \cite{Takahashi}, and later in \cite{Castro1}, this analysis has been carried out for the  family of least energy solutions. Note that these solutions satisfy the condition
$$p\displaystyle\int_{\Omega}(|\nabla u_p|^2+u_p^2)\ dx
\rightarrow 2\pi e, \quad \hbox{as }p\rightarrow +\infty,$$ which is
a particular case of \eqref{energylimit}. In \cite{Takahashi} it was proved that least energy solutions remain bounded uniformly in $p$, and  develop one peak on the boundary, whose  location is controlled by the Green's function $G$ for the Neumann problem
\begin{equation}\label{Greenequation}
\left\{\begin{array}{lr}\Delta_x G(x,y)= G(x,y)\qquad  \mbox{ in }\Omega,\\
\frac{\partial G}{\partial \nu_x}(x,y)=\delta _y(x) \mbox{ on
}\partial \Omega,
\end{array}\right.
\end{equation}
 $y\in \partial \Omega$. Indeed the concentration point turns out to satisfy $\nabla_{\tau(x_{0})}R(x_{0})=0$, where $\tau(x_{0})$ denotes a tangent vector at the point $x_{0}\in\partial\Omega$, the Robin function is defined as $R(x):=H(x,x)$, where $H$ is the regular part of $G$:
\begin{equation}\label{regularpart}
H(x, y) := G(x, y) - \frac{1}{\pi}\log \frac{1}{ |x - y|} .
\end{equation}

Later Castro
\cite{Castro1} identified a limit problem by showing that a suitable
scaling of the least energy solutions converges in $C^1_
{loc}(\overline{\mathbb{R}^2_+})$ to the regular solution 
\begin{equation}\label{v0}
U(t_1,t_2)=\log\left(\frac{4}{t_1^2+(t_2+2)^2}\right)
\end{equation}
of the
Liouville problem
\begin{equation}\label{Liouvilleproblem}
\left\{\begin{array}{lr}\Delta U= 0\qquad  \mbox{ in }\mathbb{R}^2_+\\
\frac{\partial U}{\partial \nu}=e^U\ \mbox{ on }\partial
\mathbb{R}^2_+\\
\int_{\partial\mathbb{R}^2_+} e^{U}=2\pi  \hbox{ and }
\sup_{\overline{\mathbb{R}^2_+}} U<+\infty.
\end{array}\right.
\end{equation}
He also proved that for least energy solutions 
\[\|u_p\|_{\infty} \rightarrow \sqrt{e}\mbox{ as
}p\rightarrow \infty,\] as it had been previously conjectured in
\cite{Takahashi}.

 \

Observe that problem \eqref{problem1} also admits  families of solutions which develop  $m$ boundary peaks as $p\rightarrow\infty$, for any integer $m\geq 1$, as proved in \cite{Castro2} and indeed, recently in \cite{F}, it has been proved that the boundary concentration behavior characterizes any family of solutions to \eqref{problem1} which satisfy the uniform energy bound \eqref{energylimit} (i.e. not only the least energy ones).

\

In order to state the results of \cite{F} we define, for a sequence $p_n\rightarrow +\infty$, the
blow-up set $\mathcal S$ of the sequence $p_nu_{p_n}$, where $u_{p_n}$ solves \eqref{problem1}, to be the subset 
\begin{equation}\label{def:blowupset}
\mathcal{S}:=\{\bar x\in\overline\Omega\ :\ \ \exists\, (x_n)_{n}\in\overline{\Omega},\ x_n\rightarrow\bar  x , \mbox{ with }\   p_nu_{p_{n}}(x_n) \rightarrow +\infty \}.
\end{equation}

 We summarize the results in \cite{F} as follows:

\begin{customthm}{I}\label{provaI}
\textit{Let $(u_p)_{p}$ be a family of solutions of
\eqref{problem1} satisfying \eqref{energylimit}. Then  there exist $C, c, \tilde c,\tilde C>0$ such that \begin{equation}
\label{boundSoluzio}c\leq\|\upp\|_{L^\infty(\overline{\Omega})}\leq C,\
\mbox{ for $p>1$}
\end{equation}
\begin{equation}
\label{boundEnergiap}\tilde c\leq p\int_{\partial\Omega} u_p^p d\sigma
\leq \tilde C, \mbox{ for $p$ large. }
\end{equation}
Furthermore  for
any sequence $p_n\rightarrow +\infty$, there
exists a subsequence (still denoted by $p_n$) such that the
following statements hold true:
\begin{enumerate}
\item There exists an integer $m\geq 1$, a finite collection of $m$ distinct points $\bar x_i \in
\partial \Omega$, $i = 1,\ldots,m$, such that the
blow-up set $\mathcal S$ of the sequence $p_nu_{p_n}$ is given by 
\begin{equation}\label{defSIntro}
\mathcal{S}=\{\bar x_1, \bar x_2,\ldots, \bar x_m\}.
\end{equation}
\item 
There exist $m$ positive constants $c_i>0$, $i = 1,\ldots,m$, such that
\[p_nu_{p_n}^{p_n} \displaystyle\stackrel{*}{\rightharpoonup}
\sum^m_{i=1}c_i\delta_{\bar x_i}\mbox{  in the sense of Radon measures on
$\partial \Omega$} 
\]
and
\begin{equation}
\label{convpup}
\lim_{n\rightarrow \infty}p_{n}u_{p_n}=\sum^m_{i=1}c_iG(. , \bar x_i)\quad\mbox{ in }C^1_{loc}(\overline{\Omega}\setminus \mathcal{S}),\ L^t(\Omega)\mbox{ and }
L^t(\partial\Omega),\   \forall\, t\in [1,+\infty),
\end{equation} 
 where
$G$ is the Green's function for the Neumann problem
\eqref{Greenequation}.
\item The points $\bar x_i, \
i=1,\ldots, m$, satisfy
\begin{equation}\label{impx_j relazioneintro}
c_i \nabla_{\tau(\bar x_i)} H(\bar x_i,\bar x_i)+\sum_{h\neq
i}c_h\nabla_{\tau(\bar x_i)} G(\bar x_i,\bar x_h)=0,
\end{equation}
where $\tau(\bar x_i)$ is a tangent vector to $\partial\Omega$ at $\bar x_i$ and $H$ is the regular part of $G$ as defined in \eqref{regularpart}.
\end{enumerate}
}
\end{customthm}

\

Theorem \ref{provaI} shows boundary concentration at a finite number of points in $\mathcal S\subset\partial\Omega$, moreover 
by \eqref{convpup} and \eqref{Gboundness} it follows that in any compact subset of $\overline\Omega\setminus\mathcal S$
\begin{equation}\label{convpupNEWW}
pu_{p}\leq C, 
\end{equation}
and so 
\begin{equation}
\label{uVaAZeroNeiCompatti}
\lim_{n\rightarrow \infty}u_{p_n} = 0 \quad\mbox{ in }C^1_{loc}(\overline\Omega\setminus\mathcal S).\end{equation}
Many questions arise from this result:
\begin{itemize}
\item
How does $u_{p_{n}}$ behave {\it close} to the points $\bar x_{i}$?\\
In particular, what is the asymptotic behavior of $\|u_{p_{n}}\|_{\infty}$?
\item
Can one compute the constants $c_{i}$ which appear at points (2) and (3 ) in Theorem \ref{provaI}?
\item 
What one can say about the total energy of $u_{p_{n}}$?
\end{itemize}

\

Looking at the asymptotic results for least energy solutions (\cite{Takahashi, Castro1}) and at the existence results of solutions with multiple concentrations points (\cite{Castro2}), it was conjectured in \cite{F} that for general solutions of \eqref{problem1} under the uniform energy assumption \eqref{energylimit} the constants $c_i$'s must be all equal and that an asymptotic  quantization of the energy must occur, more precisely it was conjectured that:
\begin{equation}c_{i}=2\pi \sqrt{e}, \quad  \mbox{ for }\, 1\leq i \leq m, \label{C1} \tag{C1}\end{equation}
\begin{equation}
p_{n}\int_\Omega (|\nabla u_{p_{n}}|^2+u_{p_{n}}^2)\,dx\to
m\cdot2\pi e;
\label{C2} \tag{C2}\end{equation}
as $n\rightarrow \infty$, and furthermore that 
\begin{equation}
\|u_p\|_{L^\infty(\overline{\Omega})}\to
\sqrt{e},\label{C3} \tag{C3}
\end{equation}
as $p\rightarrow +\infty$.

\

Here we answer these questions, proving in particular \eqref{C1}, \eqref{C2} and \eqref{C3}.

\

\begin{theorem}
\label{teo:Positive}
 Let $(u_p)_{p}$ be a family of 
 solutions
to \eqref{problem1} satisfying \eqref{energylimit} and let 
$p_n\rightarrow +\infty$, as $n\rightarrow \infty$, be the subsequence  such that the statements in Theorem \ref{provaI} hold true.
Then 
\begin{itemize}
\item[$(i)$] 
  \[c_{i}=\lim_{\delta\to0}\lim_{n\to+\infty}p_n\int_{B_\delta(\bar x_i)\cap\partial\Omega}u_{p_n}^{p_n}dx=2\pi \sqrt{e}, \quad  \mbox{ for }\, 1\leq i \leq m; \]
\item[$(ii)$] 
 \[\lim_{\delta\rightarrow 0}\lim_{n\rightarrow \infty}\|u_{p_{n}}\|_{ L^{\infty}(B_{\delta}(\bar x_{i})\cap \overline\Omega) }=\sqrt e\qquad \forall \, i=1,\ldots, m,\]
where $B_\delta(\bar x_i)$ is a ball of center at $\bar x_i$ and radius $\delta > 0$;
\item[$(iii)$] \[\lim_{n\rightarrow \infty}p_{n}\int_\Omega (|\nabla u_{p_{n}}|^2+u_{p_{n}}^2)\,dx= m\cdot 2\pi e;\]
\item[$(iv)$] let $\delta >0$ be such that $B_{2\delta}(\bar x_{i})\cap B_{2\delta}( \bar x_{j})=\emptyset$ for $i\neq j$ and let   
 $({{y}}_{i,n})_{n}\subset\overline{ B_{\delta}(\bar x_{i})\cap\Omega}$, $i=1,\ldots, m$, be the sequences of local maxima of $u_{p_{n}}$ around $\bar x_i$, namely
\[ u_{p_{n}}({{y}}_{i,n}):=\|u_{p_{n}}\|_{L^{\infty}(B_{\delta}( \bar x_{i})\cap \overline\Omega) },\]
then $({{y}}_{i,n})_{n}\subset\partial\Omega,$ $\lim_{n\rightarrow \infty}|{{y}}_{i,n}-\bar x_i|=0$ and, setting  
 $\mu_{i,n}:=\left(p_{n} u_{p_{n}}({{y}}_{i,n})^{p_{n}-1}\right)^{-1}(\rightarrow 0)$, then
\[
w_{i,n}(t):=\frac{p_{n}}{u_{p_{n}}({{y}}_{i,n})}\bigg(u_{p_{n}}\big(\Psi_i^{-1}( b_{i,n}+\mu_{i,n}
t)\big)-u_{p_{n}}({{y}}_{i,n})\bigg),
\]
where $b_{i,n}=\Psi_i({{y}}_{i,n})$,   $t\in T_{n}:=\{t\in \mathbb R^2\,:\,b_{i,n}+\mu_{i,n}t\in\Psi_i(\overline{\Omega}\cap
{B_{R_i}(\bar x_{i}}))\}$ and $\Psi_i$ is a change of coordinates which flattens $\partial\Omega$ near $\bar x_{i}$ and $R_{i}>0$ is a suitable radius 
(see Subsection \ref{sectionchange}).\\
Then 
\[\lim_{n\rightarrow \infty}w_{i,n}=U\quad\mbox{ in }C^1_{loc}(\overline{\R^2_+}),\] where $U$
is the solution \eqref{v0} of the Liouville problem \eqref{Liouvilleproblem}.
\end{itemize}
\end{theorem}

Theorem \ref{teo:Positive}  shows that the conjectures \eqref{C1} and \eqref{C2}  are true, furthermore points $(ii)$ and $(iv)$  provide information on the  solutions close to the concentration points $\bar x_i$ for $p$ large, in particular we identify the same limit profile $U$ around each concentration point. We stress that in  \cite{F} only the existence of a \textit{first bubble} $U$ was proved, scaling the solution around the sequence of global maxima, while the behavior around the other concentration points was unknown.
Observe that  $U$ is the same profile describing the least energy solutions  (for which $m=1$, see  \cite{Castro1}), and indeed our theorem, combined with the results in \cite{F} (Theorem \ref{provaI}), extends to  general families of solutions the asymptotic results proved in \cite{Takahashi,Castro1} for least energy solutions, thus giving a complete characterization of the asymptotic behavior for problem \eqref{problem1}. We remark that the number $m$ of concentration points coincides with the maximal number $k$ of bubbles $U$ which may appear as limit profiles (for details see Proposition \ref{thm:x1N} and \eqref{N=k} in Proposition \ref{proposizione:quasiQuanteConclusione}).
\\
We stress that from \eqref{impx_j relazioneintro} and point {\sl (i)} in Theorem \ref{teo:Positive} we also deduce that the concentration $m$-tuple $(\bar x_{1},\ldots, \bar x_{m})\in\partial\Omega$ is a critical point of the function $\varphi_{m}:(\partial\Omega)^{m}\to \mathbb R$
\begin{equation}
\label{impx_j relazioneintroNEW}\varphi_{m}(x_{1},\ldots, x_{m}):=\sum_{i=1}^{m}H(x_{i},x_{i})+\sum_{i\neq h}^{m}G(x_{i},x_{h}).
\end{equation}

We point out that  \eqref{uVaAZeroNeiCompatti} and $(ii)$-Theorem \ref{teo:Positive} clearly imply  that also conjecture \eqref{C3} holds true:
\begin{corollary}
 Let $(u_p)_{p}$ be a family of 
 solutions
to \eqref{problem1} satisfying \eqref{energylimit}.
Then 
\begin{equation}
\label{theorem3} 
\displaystyle\lim_{p\rightarrow+\infty}\|u_p\|_{\infty}=\sqrt{e}.\end{equation}
\end{corollary}

\

It is worth to remark the interesting analogy between the results here obtained for the Neumann problem \eqref{problem1} and those known for the Lane-Emden equation under Dirichlet boundary condition
\begin{equation}\label{problem111}
\left\{\begin{array}{lr}\Delta u= |u|^{p-1}u &\qquad  \mbox{ in }\Omega\\
u>0&\qquad  \mbox{ in }\Omega\\
 u=0&\ \mbox{ on }\partial \Omega.
\end{array}\right.
\end{equation}
The asymptotic behavior as $p\rightarrow +\infty$ of families $(u_{p})_{p}$ of solutions of \eqref{problem111}, under the assumption that condition \eqref{energylimit} holds, is well understood after the works \cite{RenWeiTAMS1994,RenWeiPAMS1996,AdiGrossi, DeMarchisIanniPacellaJEMS,DeMarchisIanniPacellaPositivesolutions, DGIP1,Thizy}, and the results established therein can be tought as the analogs of Theorem \ref{provaI} and Theorem \ref{teo:Positive}.
In particular it is known that $u_{p}$ stays uniformly bounded and that, up to subsequences,  peaks-up as $m$ points in the domain $\Omega$ (\cite{DeMarchisIanniPacellaJEMS}). Furthermore, it is proved  (\cite{DeMarchisIanniPacellaPositivesolutions, DGIP1, Thizy})  that   \eqref{theorem3} holds and that the concentration appears at a critical point of the functional  \eqref{impx_j relazioneintroNEW},  now defined on $\Omega^{m}$, where $G$ and $H$ are respectively  Green's and Robin's functions of $-\Delta$ in $\Omega$ under Dirichlet boundary conditions. Moreover  there is quantization of the energy, since 
\[\lim_{n\rightarrow \infty}p_{n}\int_\Omega (|\nabla u_{p_{n}}|^2+u_{p_{n}}^2)\,dx= m\cdot 8\pi e,\]
and  limit profiles  are identified. 

\

In this work we perform a blow-up analysis for the solutions of problem \eqref{problem1} following the approach developed in \cite{DeMarchisIanniPacellaPositivesolutions,DGIP1} in the framework of the Lane-Emden Dirichlet problem \eqref{problem111}. Of course one has to be very careful since now we have a boundary concentration
phenomenon due to the Neumann boundary condition, while the concentration for problem \eqref{problem111} is in $\Omega$. 
\\

We prove Theorem \ref{teo:Positive} by first performing an exhaustion method which provides a construction of concentration points. This approach relies on the energy bound assumption \eqref{energylimit} and comes with pointwise estimates of the solutions and with the description of the local asymptotic profile $U$. Similar methods have been exploited for more general $2$D Dirichlet problems (see \cite{Druet, DeMarchisIanniPacellaJEMS}), also in higher dimension (see for instance \cite{SantraWei, DruetHebeyRobert}). We have adapted the construction to deal with the Neumann problem, taking advantage also of the results in \cite{F}, this part can be found in Section \ref{sectionGeneralAsy}.
\\

Afterwards, in Section \ref{section:positive}, we refine the asymptotic analysis, showing that one can actually scale the solutions around local maxima and deriving  the sharp constants and the energy quantization. These proofs rely on a detailed local blow-up analysis, in particular we use a local Pohozaev identity (see the proof of Lemma \ref{propbeta}), pointwise estimates of the rescaled functions (see Lemma \ref {lem44}) and exploit the Green representation formula for the solutions to \eqref{problem1} (see the proof of Proposition \ref{valoremi}). Finally, at the end of Section \ref{section:positive}, we complete the proof of Theorem \ref{teo:Positive}. \\

We have postponed to Appendix  \ref{SectionAppendix} some technical estimates used throughout the paper.

\

\section{Notations} \label{notations} 
 We list here some notations used throughout the paper. First the coordinates of a point will be denoted as follows: $x=(x\coorduno, x\coorddue)\in\mathbb R^2$.\\
Newt we denote the open ball centered at a point $q=(q\coorduno,q\coorddue)\in\mathbb R^2$ and radius  $r>0$ as $B_r(q):=\{x\in \mathbb R^2\ : \ |x-q|<r\}$.
 We also define the open half ball as
 \begin{equation}\label{defhalfball}
 B_r^+(q):=B_r(q)\cap \{x\in \mathbb R^2\ : \ x\coorddue>q\coorddue\},
 \end{equation}
 its flat boundary as
 \begin{equation}\label{defflatbound}
D_r(q):=B_r(q)\cap \{x\in \mathbb R^2\ : \ x\coorddue=q\coorddue\}
\end{equation}
and its curved boundary as
 \begin{equation}\label{defcurvbound}
S_r(q):=\{x\in \mathbb R^2\ : \ |x-q|=r,\quad x\coorddue>q\coorddue\}.
\end{equation}
Moreover $\dist (x,\partial\Omega)=\inf_{y\in\partial\Omega}|x-y|$. We stress that $C$ will be a positive constant which can change from line to line.

\

\

\subsection{Change of coordinates which straightens out $\partial\Omega$ near a point on $\partial \Omega$}\label{sectionchange}
$\,$\\
We assume that $\partial\Omega\in C^2$. 
We fix a point on $\partial\Omega$ that we denote by $Q\in\partial\Omega$, in the following the change of coordinates defined below will be applied around the points in $ S=\{\bar x_{1}, \ldots, \bar x_{m}\}$ (see Theorem \ref{teo:Positive}) and around limit points of suitable special sequences (see Section \ref{sectionGeneralAsy}).

It can be proved  that there exist $R>0$ and a $C^2$ function $\rho:\mathbb R\rightarrow \mathbb R$ such that, up to reordering the coordinates and reorienting the axis
\begin{eqnarray*}
	\Omega\cap B_{R}(Q)=\{x=(x\coorduno,x\coorddue)\in B_{R}(Q)\ : \ x\coorddue>\rho(x\coorduno) \}\\
	\partial\Omega\cap B_{R}(Q)=\{x=(x\coorduno,x\coorddue)\in B_{R}(Q)\ : \ x\coorddue=\rho(x\coorduno) \}.
\end{eqnarray*}
Furthermore, up to a suitable translation of the axis we can assume that 
\[Q=(0,0)\]
so that \[\rho(0)=0\] and, up to a suitable 
rotation of the axis, we can also  assume that 
\[\rho'(0)=0.\]
We define the map
\begin{equation}\label{changement}
y=\Psi (x)
\end{equation} defined by
\[\left\{\begin{array}{ll}
&y\coorduno=x\coorduno\\
&y\coorddue=x\coorddue-\rho(x\coorduno)
\end{array}\right.\]

then $\Psi$ is one-to-one and $det (D{\Psi})=1$. Note that $\Psi$ is a $C^{2}$ function which maps $\overline\Omega\cap  B_{R}(Q)$ into a subset of the the half-plane, more precisely  
\begin{eqnarray*}\Omega_Q:=\Psi(\Omega\cap  B_{R}(Q))\subset\{y=(y\coorduno,y\coorddue)\ : y\coorddue>0\}\\
	\partial^F \Omega_Q:=\Psi(\partial \Omega\cap  B_{R}(Q))\subset\{y=(y\coorduno,y\coorddue)\ : y\coorddue=0\}
\end{eqnarray*} 
and the point $Q=0$ is mapped to the origin.

\

We define
\begin{equation}\label{defutilde}
\widetilde u_{p}(y):=u_{p}(\Psi^{-1} (y)), \quad\mbox{for all }y\in \Omega_{Q}\cup \partial^F \Omega_{Q}
\end{equation}
then (see \cite{F})
\[\left\{
\begin{array}{lr}
\Delta \widetilde u_{p}- \widetilde u_{p}- 2\rho' (y\coorduno)\frac{\partial^{2}\widetilde u_{p}}{\partial y\coorduno\partial y\coorddue}-\rho''(y\coorduno)\frac{\partial\widetilde u_{p}}{\partial y\coorddue}+(\rho'(y\coorduno))^{2}\frac{\partial^{2}\widetilde u_{p}}{\partial (y\coorddue)^{2}}=0&\qquad\mbox{ in }\Omega_{Q}
\\
\frac{\partial\widetilde u_{p}}{\partial y^{2}}\left[-1+\rho'(y\coorduno)-(\rho'(y\coorduno))^{2} \right]=\widetilde u_{p}^p &\qquad\mbox{ in }\partial^F \Omega_{Q}\end{array}
\right.\]

\

Let $x_{p}\in\overline\Omega$  be a family of points  such that $Q=\lim_{p\rightarrow +\infty}x_{p}$ and     $\mu_{p}:= \Big(pu_p(x_{p})^{p-1}\Big)^{-1}\rightarrow 0$.
\

Hence for $p$ large $x_{p}\in B_{R}(Q)$ and so  the point \[ q_{p}:=\Psi(x_{p})\in \Omega_{Q}\cup\partial^F \Omega_{Q}\]
 is well defined.

We scale $\widetilde u_{p}$ around $q_p$, setting
\[z_{p}(t):=\frac{p}{\widetilde u_{p}(q_p)}\left(\widetilde u_{p}(q_p+\mu_{p}t)-\widetilde u_{p}(q_p) \right), \quad\mbox{ for } t\in T_{Q,p}\cup \partial^F T_{Q,p}\]
where \begin{eqnarray}&& T_{Q,p}:=\{t=(t\coorduno,t\coorddue)\in\mathbb R^{2}\ : \  q_p+\mu_{p}t\ \in\Omega_{Q}\}\label{Tp}\\
&&\partial^F T_{Q,p}:=\{t=(t\coorduno,t\coorddue)\in\mathbb R^{2}\ :\   q_p+\mu_{p}t\ \in\partial^F \Omega_{Q}\}\label{boundaryTp}\end{eqnarray}

Let us observe that we can choose $\bar R>0$ such that 
\begin{equation}
\label{RestringoInsiemeRiscalate}
B_{\frac{\bar R}{\mu_{p}}}(0)\cap\{t\ : \ t\coorddue > -\frac{q_p\coorddue}{\mu_{p}}\}\subset T_{Q,p} \end{equation}
and 
\begin{equation}
\label{RestringoInsiemeRiscalate2}
B_{\frac{\bar R}{\mu_{p}}}(0)\cap\{t\ : \ t\coorddue = -\frac{q_p\coorddue}{\mu_{p}}\}\subset  \partial^F T_{Q,p}\end{equation}
Indeed, let us fix $\bar R>0$ so that 
\[
B^{+}_{2\bar R}(0)\subset \Omega_{Q}\qquad \left(\mbox{and }D_{2\bar R}(0):=B_{2\bar R}(0)\cap \{y\coorddue=0\}\subset \partial^F \Omega_{Q}\right).
\] 
Since $q_p\rightarrow 0$ (because $x_{p}\rightarrow Q$) then,  for $p$ large, 
$|q_p|\leq \bar R/2$ hence
\[
B_{\bar R}(q_p)\cap \{y\coorddue\geq 0\}\subset B_{2\bar R}(0)^{+}\cup D_{2\bar R}
(0).\]
\eqref{RestringoInsiemeRiscalate} and \eqref{RestringoInsiemeRiscalate2}  follow observing that
\[q_p+\mu_{p}t\in B_{\bar R}(q_p)\cap \{y\coorddue\geq 0\}\ \Leftrightarrow \ \left\{\begin{array}{lr} 
|t|\leq \frac{\bar R}{\mu_{p}}\\
t\coorddue\geq -\frac{q_p\coorddue}{\mu_{p}}\\
\end{array}
\right..\] 
The function $z_p$ satisfies
\[\left\{
\begin{array}{lr}
L_{p} z_p-(\mu_{p})^2 z_p=p(\mu_{p})^2 &\mbox{ in }B_{\frac{\bar R}{\mu_{p}}}(0)\cap\{t\ : \ t\coorddue > -\frac{q_p\coorddue}{\mu_{p}}\} \\
N_{p} z_p=\left(1+\frac{z_p}{p} \right)^{p} &\mbox{ in }B_{\frac{\bar R}{\mu_{p}}}(0)\cap\{t\ : \ t\coorddue= -\frac{q_p\coorddue}{\mu_{p}}\}
\end{array}\right.
\]
where, since $\rho'(0)=0$ and $\rho''$ is continuous: \[L_{p}:=\Delta-2\rho'(q_{p}\coorduno+\mu_{p}t\coorduno)
\frac{\partial^2}{\partial t\coorduno\partial t\coorddue}-\mu_{p}\rho''(q_{p}\coorduno+\mu_{p}t\coorduno)
\frac{\partial}{\partial t\coorddue} 
+
[\rho'(q_{p}\coorduno+\mu_{p}t\coorduno)]^2\frac{\partial^2}{\partial (t\coorddue)^2}\ \underset{p\rightarrow +\infty}{\longrightarrow} \ \Delta\]
and
\[N_{p}:=\left[-1+\rho'(q_{p}\coorduno+\mu_{p}t\coorduno)-[\rho'(q_{p}\coorduno+\mu_{p}t\coorduno)]^{2} \right]\frac{\partial}{\partial t\coorddue }\ \underset{p\rightarrow +\infty}{\longrightarrow}\ \frac{\partial}{\partial\nu}.\] 
Thanks to these convergences one can restrict to consider the case when $\partial \Omega$ is flat near $Q$, since the same arguments adapt to the non-flat case   (see for instance \cite{F, Castro1}).
   \

\section{Exhaustion of concentration points}\label{sectionGeneralAsy}

Given a family $(u_p)$ of solutions of \eqref{problem1}, for a sequence $p_n\rightarrow +\infty$ we 
define the {\it concentration set} $\widetilde{\mathcal{S}}$ of $u_{p_n}$ as 
\begin{equation}\label{S}
\widetilde{\mathcal{S}}:=\left\{ x\in\overline\Omega\ : \ \exists \,(x_n)_n\in\overline\Omega,\ x_n\rightarrow x, \mbox{ with }\ p_nu^{p_n-1}_{p_n}(x_n)\rightarrow +\infty\right\}\subset\overline{\Omega}.
\end{equation}
Clearly  
\begin{equation}\label{FirstInclusion}
\widetilde{\mathcal{S}}\subseteq \mathcal S =\{\bar x_{1},\ldots,\bar x_{m}\}\ (\subset\partial\Omega),
\end{equation}
where   $\mathcal S$ is the blow-up set of the sequence $p_{n}u_{p_{n}}$ (see \eqref{def:blowupset} for the definition) characterized in \cite{F} (see Theorem \ref{provaI} in the Introduction). Indeed by the definition of $\widetilde{\mathcal{S}}$, for any $x\in \widetilde{\mathcal{S}}$  
there exists a sequence $x_{n}\in\overline\Omega,\ x_{n}\rightarrow x$ such that \[
p_{n}u_{p_{n}}^{p_{n}-1}(x_{n})\rightarrow +\infty,
\] 
then clearly \[p_n u_{p_n}(x_{n})\rightarrow +\infty,\]  hence,  by the definition of $\mathcal S$, $x\in\mathcal S$ and \eqref{FirstInclusion} is proved.\\ Let us also note that \[\liminf_{n\rightarrow \infty} u_{p_n}(x_{n})\geq 1.\]
\\
As a consequence, up to reordering the points in $\mathcal S$, there exists $N\leq m$ such that 
\begin{equation}
\label{PrimaCaratterizzazioneSTilde}
\widetilde{\mathcal{S}}=\{\bar x_{1},\ldots, \bar x_{N}\}.
\end{equation}

In this section we will prove the existence of a maximal number $k$ of \textit{concentrating sequences} $x_n$ for the set $\widetilde{\mathcal S}$, satisfying specific properties, in particular we get pointwise estimates and a description of $u_p$ close to the points of $\widetilde{\mathcal S}$. The main result is contained in Proposition \ref{thm:x1N} below.

\

We introduce some notation. For  $l\in\N\setminus\{0\}$ families of
points $(\xip)_p\subset\overline\Omega$, $i=1,\ldots,l$, such that
\begin{equation}
\label{muVaAZero} p\upp^{p-1}(\xip)\to+\infty\ \mbox{ as }\
p\to+\infty,
\end{equation}
we define the parameters 
\begin{equation}\label{mip} \mip:=\left(p
\upp^{p-1}(\xip)\right)^{-1}\ \ (\rightarrow 0,\ \mbox{ as }p\rightarrow +\infty), 
\end{equation} 
and introduce the following properties:
\begin{itemize}
\item[$(\mathcal{P}_1^l)$] For any $i,j\in\{1,\ldots,l\}$, $i\neq j$,
\[
\lim_{p\to+\infty}\fr{|\xip-\xjp|}{\mip}=+\infty.
\]
\item[$(\mathcal{P}_2^l)$] For any $i\in\{1,\ldots,l\}$,
\[
\lim_{p\to+\infty}\fr{dist(\xip,\partial \Omega)}{\mip}= 0.
\]
\item[$(\mathcal{P}_3^l)$]
 For any $i=1,\ldots,l$, let $Q_i\in\partial\Omega$ be such that \[Q_i:=\lim_{p\rightarrow +\infty} x_{i,p},\] let $\Psi_i$ be the change of coordinates which straightens $\partial\Omega$ in a neighborhood of $Q_i$ of radius $R_i>0$, let $q_{i,p}=\Psi_i(x_{i,p})$ and let  
\begin{equation}\label{vipdef}
z_{i,p}(t):=\fr{p}{\upp(\xip)}\bigg(\upp\big(\Psi_i^{-1}( q_{i,p}+\mip
t)\big)-\upp(\xip)\bigg)\quad\text{for $t\in T_{i,p}\cup \partial^F T_{i,p}$,}
\end{equation}
where $T_{i,p}:=T_{Q_i,p}$, see \eqref{Tp}-\eqref{boundaryTp} in Section \ref{notations} for the notations.\\
Then 
\begin{equation}\label{vip}
T_{i,p}\cup \partial^F T_{i,p}\to \overline{\R^2_+}\qquad\text{and}\quad z_{i,p}(t)\longrightarrow U(t)\qquad\text{in $C^1_{loc}(\overline{\R^2_+})$ as $p\to+\infty$,}
\end{equation} where $U$ is the function in \eqref{v0}.
\item[$(\mathcal{P}_4^l)$] There exists $C>0$ such that
\[
p R_{l,p}(x) \upp^{p-1}(x)\leq C
\]
for all $p>1$ and all $x\in \overline{\Omega}$.
where $R_{l,p}$ is the function 
\begin{equation}\label{RNp} R_{l,p}(x):=\min_{i=1,\ldots,l}
|x-\xip|, \ \forall x\in\overline{\Omega}.
\end{equation}
\end{itemize}

\begin{remark}
	\label{boundaryconcpoints}
	If we assume that there exist $l\in\N\setminus\{0\}$ families of
	points $(\xip)_p\subset\overline\Omega$, $i=1,\ldots,l$ which satisfies \eqref{muVaAZero} 
	and such  that property  $(\mathcal{P}_4^l)$ holds true,
	then it is clear that the concentration set defined in \eqref{S} reduces to \[
	\widetilde{\mathcal{S}}=\left\{\lim_{p\to+\infty}\xip,\,i=1,\ldots,l\right\}.
	\] 
\end{remark}

\begin{lemma}\label{lemma:BoundEnergiaBassino}
If there exists  $l\in\N\setminus\{0\}$ such that the properties
$(\mathcal{P}_1^l)$, $(\mathcal{P}_2^l)$ and $(\mathcal{P}_3^l)$
hold for families $(\xip)_{i=1,\ldots,l}$ of points satisfying
\eqref{muVaAZero}, then
\[
p\int_\Omega (|\na\upp|^2+u_p^2)\,dx\geq2\pi\sum_{i=1}^l
\alpha_i^2+o_p(1)\ \mbox{ as }p\rightarrow +\infty,
\]
where  $\alpha_i:=\liminf_{p\to+\infty}\upp(\xip)\
(\geq 1,\ \mbox{ by \eqref{muVaAZero}})$.

\end{lemma}
\begin{proof}
	Let us fix $i\in\{1,\ldots,l\}$. 	Since $\lim_{p\to+\infty} x_{i,p}=Q_i\in\partial\Omega$ and $\lim_{p\to+\infty}\mu_{i,p}=0$, then for any $R>0$ and $p$ sufficiently large $B_{R\mu_{i,p}}(x_{i,p})\cap\partial\Omega\subset B_{\delta_i}(Q_i)$.\\
	As in Subsection \ref{sectionchange} we can assume w.l.o.g. that $Q_{i}=0$.
	We claim that
	\begin{equation}\label{claim}
	(B_{\tfrac {R\mu_{i,p}}3}( q_{i,p})\cap \{y\coorddue=0\})\subset \Psi(B_{R\mu_{i,p}}(x_{i,p})\cap \partial \Omega),
	\end{equation}
	where $q_{i,p}:=\Psi (x_{i,p})$.
Given $\Psi$ as in	\eqref{changement}, it follows that $\Psi^{-1}$ is a $C^2$ function in a neighbourhood of $(0,0)=\Psi(Q_i)$ furthermore 
$D\Psi^{-1}(0,0)=I$. Thus
\begin{equation}
	\label{picche0}
	\exists\,\delta>0\quad\text{such that}\quad 
	\|D\Psi^{-1}(y)\|\leq 3
	\quad\forall\, y\in \overline{B^+_{\delta}(0,0)}.
	\end{equation}
Since $\lim_{p\rightarrow +\infty }q_{i,p}=\Psi(Q_i)=(0,0)$, then for $p$ sufficiently large we have
\[(B_{\tfrac {R\mu_{i,p}}3}( q_{i,p})\cap \{y\coorddue=0\})\subset \overline{B^+_{\delta}(0,0)}.\]
Let $y=(y\coorduno,0)\in(B_{\tfrac {R\mu_{i,p}}3}(q_{i,p})\cap \{y\coorddue=0\})$ then $y=\Psi(x)$ where 
\begin{eqnarray*}
|x-x_{i,p}|&=&|\Psi^{-1}(y)-\Psi^{-1}(q_{i,p})|
\\
&\leq& \sup_{\overline{B^+_{\delta}(0,0)}}\|D\Psi^-1\||y-q_{i,p}|
\\
&\overset{\eqref{picche0}}{\leq} & R\mu_{i,p}.
\end{eqnarray*}
This proves \eqref{claim}.\\
		Let us write, for any $R>0$, recalling the definition of $\widetilde u_p$ in \eqref{defutilde}
	\begin{eqnarray}\label{appo1}
	p\int_{B_{R\mu_{i,p}}(\xip)\cap\partial \Omega}
	\upp^{p+1}\,d\sigma(x) &\geq&p\int_{\Psi_i(B_{R\mu_{i,p}}(\xip)\cap\partial
		\Omega)}
	\widetilde{u}^{p+1}_p(y)\,d\sigma(y)
 \nonumber\\
	&\stackrel{\eqref{claim}}\geq&p\int_{B_{\tfrac {R\mu_{i,p}}3}(q_{i,p})\cap \{y\coorddue=0\}}
	\widetilde{u}_p^{p+1}(y)\,d\sigma(y)
 \nonumber \\
	&\geq&p\mu_{i,p}\int_{B_{\tfrac {R}3}(0)\cap \{t\coorddue=-\frac{q_{i,p}\coorddue}{\mu_{i,p}}\}}
	\widetilde{u}^{p+1}_p(q_{i,p}+\mip t)\,dt\coorduno 
\nonumber\\
	&
	\stackrel{\eqref{mip}}\geq&\upp(\xip)^2\int_{B_{\tfrac {R}3}(0)\cap \{t\coorddue=-\frac{q_{i,p}\coorddue}{\mu_{i,p}}\}}\fr{\widetilde{u}^{p+1}_{p}(q_{i,p}+\mip
		t)}{\upp^{p+1}(\xip)}
	\,dt\coorduno \nonumber\\
	&\geq& \upp(\xip)^2\int_{B_{\tfrac {R}3}(0)\cap \{t\coorddue=-\frac{q_{i,p}\coorddue}{\mu_{i,p}}\}}\left(1+\fr{z_{i,p}(t)}{p}\right)^{p+1}\,dt\coorduno .
	\end{eqnarray}
	Thanks to $(\mathcal{P}_3^l)$, we have \bel\label{appo2}
	\|z_{i,p}-U\|_{L^{\infty}(\overline{B_{R/3}^+})}=o_p(1)\ \mbox{ as
	}p\rightarrow +\infty. \eel Thus by \eqref{appo1}, \eqref{appo2} and
	Fatou's lemma \bel\label{appo3}
	\liminf_{p\to+\infty}\,\left(p\int_{B_{R\mu_{i,p}}(\xip)\cap
		\partial \Omega}\upp^{p+1}\,d\sigma(x)\right)\geq\alpha^2_i\int_{B_{R/3}(0)\cap\{t\coorddue=0\}} e^{U(t)}\,
	dt\coorduno . \eel
	
	Moreover by virtue of $(\mathcal{P}_1^l)$ it is not hard to see that
	$B_{R\mip}(\xip)\cap B_{R\mu_{j,p}}(x_{j,p})=\emptyset$ for all
	$i\neq j$. Hence, in particular, thanks to \eqref{appo3}
	\[
	\liminf_{p\to+\infty}\,\left(p\int_{\partial\Omega}
	\upp^{p+1}\,d\sigma(x)\right)\geq \sum_{i=1}^l
	\left(\alpha^2_i\int_{B_{R/3}(0)\cap\{t\coorddue=0\}} e^{U(t)}\, dt\coorduno \right).
	\]
	At last, since this holds for any $R>0$, we get
	\[
	p\int_\Omega(|\na\upp|^2+u_p^2)\,dx=p\int_{\partial\Omega}
	\upp^{p+1}\,d\sigma(x)\geq \sum_{i=1}^l \alpha_i^2
	\int_{\partial\R^2_+}e^{U(t)}\,dt\coorduno  +o(1)=2\pi\sum_{i=1}^l
	\alpha_i^2+o(1),  \] as
	$p\rightarrow +\infty$.
\end{proof}

\

Using an exhaustion method, we establish the existence of a maximal number $k$ of
``bubbles'' $U$ appearing about the points of the boundary subset $\widetilde S$.
\begin{proposition}\label{thm:x1N}
Let $(\upp)$ be a family of solutions to \eqref{problem1} and assume
that \eqref{energylimit} holds. Then after passing to a subsequence $p_{n}\rightarrow +\infty$ as $n\rightarrow \infty$, there exist an integer 
$k\geq 1$ and $k$ families of points $(x_{i,p_{n}})$ in
$\overline{\Omega}$  $i=1,\ldots, k$ such that $(\mathcal{P}_1^k)$, $(\mathcal{P}_2^k)$,
$(\mathcal{P}_3^k)$ and $(\mathcal{P}_4^k)$ hold. Moreover given any family of points $x_{k+1,p_{n}}$, it is
impossible to extract a new sequence from the previous one such that
$(\mathcal{P}_1^{k+1})$, $(\mathcal{P}_2^{k+1})$,
$(\mathcal{P}_3^{k+1})$ and $(\mathcal{P}_4^{k+1})$ hold with the
sequences $(x_{i,p_{n}})$, $i=1,\ldots,k+1$. 
Furthermore, there exists  $N\leq\min\{m, k\}$ (where $m$ is the number of points of the set $\mathcal S$) such that, up to reordering the points $\bar x_{i}\in\mathcal S$, it holds
 \begin{equation}
\label{PrimaCaratterizzazioneSTildeN}\widetilde{\mathcal{S}}=\left\{\lim_{n\to+\infty}x_{i, p_n},\,i=1,\ldots,k\right\}=\{\bar x_{1},\ldots, \bar x_{N}\},
\end{equation}
where $\widetilde{\mathcal S}$ is the concentration set defined in \eqref{S}.
\end{proposition}
\begin{remark}
The point $x_{1,p_{n}}$ can be taken to be a maximum point of $u_{p_{n}}$ in $\overline\Omega$, hence it belongs to $\partial \Omega$, see STEP 1 below. The other sequences $x_{i,p}$, $i=2,\ldots, k$ are instead in $\overline\Omega$. \\
Observe also that the number $N$ of distinct points in $\widetilde{\mathcal{S}}$ satisfies $N\leq k$.
\end{remark}
\begin{proof}
For simplicity throughtout the proof we will denote any sequence $p_{n}\rightarrow +\infty$ as $n\rightarrow \infty$ simply by $p$.\\

\textit{STEP 1. We show that there exists a family $(x_{1,p})$ of
points in $\Omega$ such that, after passing to a sequence
$(\mathcal{P}^1_2)$ and $(\mathcal{P}^1_3)$ hold.\\}

Let us choose $x_{1,p}$ be a point in ${\overline{\Omega}}$ where
$\upp$ achieves its maximum. In \cite{F} it has been proved that $x_{1,p}\in\partial \Omega$ and that it satisfies  $(\mathcal{P}^1_3)$. 

\

\textit{STEP 2. We assume that $(\mathcal{P}_1^n)$, $(\mathcal{P}_2^n)$
and $(\mathcal{P}_3^n)$ hold for some $n\in\N\setminus\{0\}$. Then
we show that either $(\mathcal{P}_1^{n+1})$, $(\mathcal{P}_2^{n+1})$
and $(\mathcal{P}_3^{n+1})$ hold or $(\mathcal{P}_4^n)$ holds,
namely there exists $C>0$ such that
$$
p R_{n,p}(x) \upp^{p-1}(x)\leq C
$$
for all $x\in\Omega$, with $R_{n,p}$  defined as in \eqref{RNp}.\\
}

\

Let $n\in\N\setminus\{0\}$ and assume that $(\mathcal{P}_1^n)$,
$(\mathcal{P}_2^n)$ and $(\mathcal{P}_3^n)$ hold while
\bel\label{Pknonvale} \sup_{x\in\overline{\Omega}}\left(p R_{n,p}(x)
\upp^{p-1}(x)\right)\to+\infty\quad\textrm{as $p\to+\infty$}. \eel

We let $x_{n+1,p}\in\overline\Omega$ be such that \bel\label{xk+1} p
R_{n,p}(x_{n+1,p})
\upp^{p-1}(x_{n+1,p})=\sup_{x\in\overline{\Omega}}\left(p
R_{n,p}(x) \upp^{p-1}(x)\right). \eel By \eqref{Pknonvale},
\eqref{xk+1} and since $\Omega$ is bounded it is clear that
$$
p\upp^{p-1}(x_{n+1,p})\to+\infty\quad\textrm{as $p\to+\infty$}
$$
and \begin{equation}\label{geq1} \displaystyle\liminf _{p\rightarrow
+\infty}u_p(x_{n+1,p})\geq 1.
 \end{equation} 
 We will prove that 
 $(\mathcal{P}_1^{n+1})$, $(\mathcal{P}_2^{n+1})$ and $(\mathcal{P}_3^{n+1})$ hold with the added sequence $(x_{n+1,p})$.\\
 
 \textit{\underline{Proof of $(\mathcal{P}_1^{n+1})$}.}\\
 
 We first claim that \bel\label{claimxk+1}
\fr{|\xip-x_{n+1,p}|}{\mu_{i,p}}\to+\infty\quad\textrm{as
$p\to+\infty$} \eel for all $i=1,\ldots,n$ and $\mu_{i,p}$ as in
\eqref{mip}.\\
Let us assume by contradiction that there exists
$i\in\{1,\ldots,n\}$ such that $|\xip-x_{n+1,p}|/\mip\to R$ as
$p\to+\infty$ for some $R\geq0$. Then the points $\xip$ and
$x_{n+1,p}$ are close to each other and by virtue of $(\mathcal{P}_2^n)$,
they are very close to the boundary of $\Omega$. Let us denote
$q_{i,p}:=\Psi_i(\xip)$ and $q_{n+1,p}:=\Psi_i(x_{n+1,p})$ where $\Psi_i$ is
 the function defined by \eqref{changement} around the boundary point $Q_i:=\lim_{p\to+\infty}x_{i,p}$. Since $\Psi_i$ is
$C^2$, $(|q_{i,p}-q_{n+1,p}|/\mip)_p$ is bounded. Up to subsequence,
$|q_{i,p}-q_{n+1,p}|/\mip\to R'$ as $p\to+\infty$ for some $R'\geq0$.
Thanks to $(\mathcal{P}_3^n)$, we get
\begin{align*}
\lim_{p\to+\infty}p|\xip-x_{n+1,p}| \upp^{p-1}(x_{n+1,p})
&=\lim_{p\to+\infty} \fr{|\xip-x_{n+1,p}|}{\mu_{i,p}}\left(\frac{\upp(x_{n+1,p})}{\upp(\xip)}\right)^{p-1}\\
&=\lim_{p\to+\infty} \fr{|\xip-x_{n+1,p}|}{\mu_{i,p}}\left(\frac{\upp(\Psi_i^{-1}(q_{n+1,p}))}{\upp(\xip)}\right)^{p-1}
\\
&=\lim_{p\to+\infty} \fr{|\xip-x_{n+1,p}|}{\mu_{i,p}}\left(1+\frac{z_{i,p}(\mu_{i,p}^{-1}(q_{n+1,p}-q_{i,p}))}{p}\right)^{p-1}\\
&=\fr{4R}{(t\coorduno)^2+(t\coorddue+2)^2}<+\infty \ (\hbox{where} \
(t\coorduno)^2+(t\coorddue)^2=R'^2),
\end{align*} 
against {\eqref{Pknonvale} and} \eqref{xk+1}, thus \eqref{claimxk+1} holds.\\
Setting
\bel\label{mk+1} \mu_{n+1,p}:=\left[p
\upp^{p-1}(x_{n+1,p})\right]^{-1}\rightarrow 0 \hbox{ as
}p\to+\infty, \eel by \eqref{Pknonvale} and \eqref{xk+1} we deduce
that \bel\label{Rkmk+1}
\fr{R_{n,p}(x_{n+1,p})}{\mu_{n+1,p}}\to+\infty\quad\textrm{as
$p\to+\infty$.} \eel
Then \eqref{claimxk+1}, \eqref{Rkmk+1} and $(\mathcal{P}_1^n)$ imply that $(\mathcal{P}_1^{n+1})$ holds with the added sequence $(x_{n+1,p})$.\\

\textit{\underline{Proof of $(\mathcal{P}_2^{n+1})$}.}\\
Let us prove that for any $S>0$
\begin{equation}\label{quotientbounded}
\sup_{B_{S\mu_{n+1,p}}(x_{n+1,p})\cap\overline\Omega}\frac{u_p(x)}{u_p(x_{n+1,p})}\leq
1+O(\frac{S\mu_{n+1,p}}{(p-1)R_{n,p}(x_{n+1,p})}).
\end{equation}
Let $x\in B_{S\mu_{n+1,p}}(x_{n+1,p})\cap\overline\Omega$,  since $x_{n+1,p}$ satisfies \eqref{xk+1},
$${R_{n,p}(x)}u_p(x)^{p-1}\leq {R_{n,p}(x_{n+1,p})} u_p^{p-1}(x_{n+1,p}).$$ 
Furthermore $|x-x_{n+1,p}|\leq
S\mu_{n+1,p}$, thus
\begin{eqnarray*}
	R_{n,p}(x)&\geq&\min_{i=1,\ldots,n}|x_{n+1,p}-\xip|-|x-x_{n+1,p}|\\
	&\geq&R_{n,p}(x_{n+1,p})-S\mu_{n+1,p}.
\end{eqnarray*}
Then, since for $p$ large by \eqref{Rkmk+1}, $R_{n,p}(x_{n+1,p})-S\mu_{n+1,p}>0$
\begin{align*}
u_p^{p-1}(x)&\leq
\frac{R_{n,p}(x_{n+1,p})}{R_{n,p}(x_{n+1,p})-S\mu_{n+1,p}}
u_p^{p-1}(x_{n+1,p})\\
&\leq \frac{1}{1-\frac{S}{R_{n,p}(x_{n+1,p})}\mu_{n+1,p}}
u_p^{p-1}(x_{n+1,p})\\
&\leq \big(1+O(\frac{S\mu_{n+1,p}}{R_{n,p}(x_{n+1,p})})\big)
u_p^{p-1}(x_{n+1,p}),
\end{align*}
thus \eqref{quotientbounded} is proved.

\

Let us now introduce the rescaled function
\begin{equation}\label{cuori}
v_{n+1,p}(t):=\frac{p}{u_p(x_{n+1,p})}[u_p(x_{n+1,p}+\mu_{n+1,p}t)-u_p(x_{n+1,p})], \forall t\in
\widetilde{\Omega}_{n+1,p}:=\mu_{n+1,p}^{-1}(\Omega-x_{n+1,p}).  
\end{equation}

Observe that by definition for  $t\in\widetilde{\Omega}_{n+1,p}\cap B_S(0)$
\begin{equation}\label{newv} v_{n+1,p}(t)=p\left(\frac{u_p(x)}{u_p(x_{n+1,p})} -1\right),
\end{equation}
where $x:=x_{n+1,p}+\mu_{n+1,p} t\in \Omega\cap B_{S\mu_{n+1,p}}(x_{n+1,p})$, hence  by  
\eqref{quotientbounded}
and 
\eqref{Rkmk+1} 
 it follows that
 for any $S>0$ one has
\begin{equation}
\label{**}
\limsup_{p\to+\infty}\sup_{\widetilde{\Omega}_{n+1,p}\cap B_S(0)}v_{n+1,p}\leq 0.
\end{equation}

\

Next to show $(\mathcal{P}_2^n)$ we argue by contradiction assuming that $\lim_{p\rightarrow+\infty} dist(x_{n+1,p}, \partial \Omega)\mu^{-1}_{n+1,p}\neq 0$. Up to a subsequence two cases may
occur:
\begin{enumerate}
	\item $dist(x_{n+1,p},\partial \Omega)\mu^{-1}_{n+1,p}\longrightarrow
	L>0$,
	\item $dist(x_{n+1,p},\partial \Omega)\mu^{-1}_{n+1,p}\longrightarrow +\infty$.
\end{enumerate}

\underline{Case $(1)$.}  Let us start by the first case. We have $x_{n+1,p}\longrightarrow
Q_{n+1}\in \partial \Omega$. We may assume without loss
of generality that the unit outward normal to $\partial \Omega$ at
$Q_{n+1}$ is $-e\coorddue$, where $e\coorddue$ is the second element of the
canonical basis of $\mathbb{R}^2$, and that $\partial \Omega$ is contained in
the hyperplane $x\coorddue=0$. For simplicity we will also assume that
$\partial \Omega$ is flat near $Q_{n+1}$, we point out that all our arguments can be adapted to the non-flat case considering the change of coordinates which straightens out $\partial\Omega$ near $Q_{n+1}$, introduced in Section \ref{sectionchange}
(see for instance \cite[Theorem 3]{Castro1}). The flatness
assumption means that the function $\Psi_{n+1}$ in \eqref{changement} is the identity, namely that there exists $R:=R_{n+1} > 0$ such that
$\Omega\cap B^+_R(Q_{n+1})
= B^+_R(Q_{n+1})$ and $\partial\Omega\cap \partial B^+_R(Q_{n+1})
= D_R(Q_{n+1})$. \\
In particular, for $p$ large one has that 
$x\coorddue_{n+1,p}=dist(x_{n+1,p},\partial \Omega)$, so that by assumption
\begin{equation}\label{riscrivoCaseL}
\frac{x\coorddue_{n+1,p}}{\mu_{n+1,p}}\longrightarrow L
\end{equation}
as $p\rightarrow +\infty$.

Let us project $x_{n+1,p}$ on  the boundary defining the point $\hat x_{n+1,p}:= (x_{n+1,p}\coorduno,0) (\in\partial\Omega)$,  and let us set
\begin{eqnarray}\label{cuoris}
s_{n+1,p}(t) &:=&  v_{n+1,p}\Big(t\coorduno, t\coorddue-\frac{x\coorddue_{n+1,p}}{\mu_{n+1,p}}\Big)
,\qquad \forall t\in
\widehat{\Omega}_{n+1,p}
:=\mu_{n+1,p}^{-1}(\Omega-\hat x_{n+1,p})
\nonumber\\
&\overset{\eqref{cuori}}{=}&\frac{p}{u_p(x_{n+1,p})}[u_p(\hat x_{n+1,p}+\mu_{n+1,p}t)-u_p(x_{n+1,p})]. 
\end{eqnarray}
We can choose $\delta>0$ such that
$B_{\delta}^{+}(\hat x_{n+1,p})\subset B^+_R(Q_{n+1})$, 
hence
\[
B_{\frac{\delta}{\mu_{n+1,p}}}^{+}(0)\subset \widehat{\Omega}_{n+1,p} \quad\mbox{ and }\quad D_{\frac{\delta}{\mu_{n+1,p}}}(0)\subset \partial\widehat{\Omega}_{n+1,p} 
\]
and, by \eqref{problem1}, the rescaled function $s_{n+1,p}$ solves the system
\begin{equation}\label{systemsn+1,p}
\left\{\begin{array}{ll} -\Delta s_{n+1,p} + \mu_{n+1,p}^2s_{n+1,p}
=-\mu_{n+1,p}^2p & \qquad \hbox{in }B^{+}_{\frac{\delta}{\mu_{n+1,p}}}(0),
\\
\displaystyle\frac{\partial s_{n+1,p}}{\partial
	\nu}=\left(1+\frac{s_{n+1,p}}{p}\right)^{p}& \qquad\hbox{on
}D_{\frac{\delta}{\mu_{n+1,p}}}(0), 
\end{array}
\right.
\end{equation}
furthermore  for any $\sigma>0$, by \eqref{riscrivoCaseL}, there exists $S>0$ such that $B^{+}_{\sigma}(0)\subset B_{S}(0,\frac{x\coorddue_{n+1,p}}{\mu_{n+1,p}})\cap \widehat\Omega_{n+1,p}$, then
\begin{eqnarray}
\label{**persn}
\limsup_{p\to+\infty}\sup_{B^{+}_{\sigma}(0)}s_{n+1,p}(t)&\leq &
\limsup_{p\to+\infty}\sup_{B_{S}(0,\frac{x\coorddue_{n+1,p}}{\mu_{n+1,p}})\cap \widehat\Omega_{n+1,p}}v_{n+1,p}(t\coorduno,t\coorddue -\frac{x\coorddue_{n+1,p}}{\mu_{n+1,p}})
\nonumber \\
&\leq&
\limsup_{p\to+\infty}\sup_{B_{S}(0)\cap \widetilde\Omega_{n+1,p}}v_{n+1,p}(t)
\overset{\eqref{**}}{\leq} 0.
\end{eqnarray}
Arguing similarly as in the proof of \cite[Lemma 2]{Castro1}, we will prove that for any $r>\frac{L}{3}$ there exist $C>0$, $p_r>1$ and $\alpha\in(0,1)$ such that 
\begin{equation}\label{boundC1alpha}
\|s_{n+1,p}\|_{C^{1,\alpha}(B_r^{\textcolor{blue}{+}}(0))}\leq C,\qquad \mbox{for any $p>p_r$.}
\end{equation}
We first observe that for any fixed $q\geq2$ and for $p$ sufficiently large
\begin{equation}\label{prima}
\int_{{B}_{4r}^+(0)}|p\mu_{n+1,p}^2|^q\,dx=o_p(1)
\end{equation}
and 
\begin{eqnarray}
\label{seconda}
\int_{D_{4r}(0)}\left|1+\frac{s_{n+1,p}(t)}{p}\right|^{pq}\,d\sigma&=&\frac{1}{\mu_{n+1,p}}\int_{D_{4r\mu_{n+1,p}}(\hat x_{n+1,p})}\left(\frac{u_p(x)}{u_p(x_{n+1,p})}\right)^{pq}\,d\sigma\nonumber\\
&\leq&\frac{p}{u_p^2(x_{n+1,p})}\sup_{x\in D_{4r\mu_{n+1,p}}(\hat x_{n+1,p}) }\left(\frac{u_p(x)}{u_p(x_{n+1,p})}\right)^{p(q-1)-1}\int_{\partial \Omega}u^{p+1}_p(x)\,d\sigma\nonumber\\
&\leq& C,
\end{eqnarray}
where in the last inequality we used \eqref{geq1}, $D_{4r\mu_{n+1,p}}(\hat x_{n+1,p})\subset B_{cr\mu_{n+1,p}}(x_{n+1,p})\cap\bar\Omega$ for some constant $c>0$ (being $r>\frac{L}{3}$), \eqref{quotientbounded}, \eqref{Rkmk+1} and the energy bound \eqref{energylimit}, since for a solution $u_p$ one has $\int_{\Omega}(|\nabla u_p|^2+u_p^2)\,dx=\int_{\partial\Omega}u_p^{p+1}d\sigma$.

Let us now consider the solution $w_p$ to
\begin{eqnarray}\label{w0}
\left\{\begin{array}{ll} -\Delta w_p+\mu_{n+1,p}^2 w_p=-p\mu_{n+1,p}^2  & \quad\hbox{ in }{B}_{4r}^+(0),\\
\displaystyle\frac{\partial w_p}{\partial
	\nu}= \left ( 1+\frac{s_{n+1,p}}{p}\right)^{p} & \quad\hbox{ on }{D}_{4r}(0),\\
w_p=0 &\quad \hbox{ on }{S}_{4r}(0).
\end{array}
\right.
\end{eqnarray}
By \eqref{prima} and \eqref{seconda}, with $q=2$, the existence of such $w_p\in H^1({B}^+ _{4r}(0))$ is guaranteed by Lax-Milgram. Furthermore arguing as in  \cite[Theorem 5.3]{Shamir2},   we have  by \eqref{prima} and \eqref{seconda}, that $w_p\in W^{\frac{1}{2}+t,q}(B_{4r}^+(0))$ with the uniform bound
\begin{equation}\label{estimationw1}
\|w_p\|_{W^{\frac{1}{2}+t,q}(B_{4r}^+(0))}\leq
C\bigg(\|\mu_{n+1,p}^2 p\|_{L^q(B^+ _{4r}(0))}
+\left\|\left(1+\frac{s_{n+1,p}}{p}\right)^p\right\|_{L^q(D
_{4r}(0))}\bigg)\leq C,
\end{equation} for $q>4$ and $0<t<2/q$.\\ 
In particular, by Sobolev embeddings, $\|w_p\|_{L^{\infty}(B_{4r}^+(0))}
\leq C$, so we can define the function 
\[\varphi_p := w_p - s_{n+1,p} +
\|w_p\|_{L^\infty(B^+ _{4r}(0))}+1,\] which solves
\[
\left\{\begin{array}{ll} 
-\Delta \varphi_{p}+\mu_{n+1,p}^2 \varphi_{p}=\mu_{n+1,p}^2
(\|w_p\|_{L^\infty(B^+ _{4r})}+1)&  \hbox{ in }B_{4r}^+(0),\\
\frac{\partial  \varphi_{p}}{\partial
\nu} =0 &\hbox{ on }D_{4r}(0),\end{array}
\right.
\]
furthermore, since  $s_{n+1,p} \leq 1 $ in $B^+_{4r}(0)$ for $p$ sufficiently large  by \eqref{**persn}, then 
 \[\varphi_{p}\geq 0\quad \ \quad \quad \quad\hbox{in }B^+_{4r}(0).\]

We define, for $t=(t_1,t_2)\in B_{4r}(0)$, the function
\[
\hat\varphi_p(t)=\left\{\begin{array}{ll}
	\varphi_p(t)&\text{if $t_2\geq0$}\\
	\varphi_p(t_1,-t_2)&\text{if $t_2<0$,}
\end{array}
\right.
\]
which turns out to be a non-negative weak solution to \[-\Delta \varphi+\mu_{n+1,p}^2\varphi=\mu_{n+1,p}^2(\|w_p\|_{L^{\infty}(B_{4r}^+(0))}+1)\qquad\mbox{ in }B_{4r}(0).\]
Therefore by the Harnack inequality we get for every $a\geq1$
\begin{eqnarray*}
\left(\fint_{B_{3r}(0)}\hat\varphi_p^a\right)^{\tfrac1a}&\leq& C\left(\inf_{B_{3r}(0)}\hat\varphi_p+\|\mu_{n+1,p}^2
(\|w_p\|_{L^\infty(B^+ _{4r}(0))}+1)\|_{L^2(B_{4r}(0))}\right)\\
&\overset{3r>L\,+\,\eqref{riscrivoCaseL}}{\leq}&C\left(\varphi_p(0,\frac{x\coorddue_{n+1,p}}{\mu_{n+1,p}})+\mu_{n+1,p}^2C\right)\\
&{\leq}&C\left(2\|w_p\|_{L^{\infty}(B_{4r}^+(0))}+1+\mu_{n+1,p}^2C\right)\\
&{\leq}&C,
\end{eqnarray*}
where we have used that $s_{n+1,p}(0,\frac{x\coorddue_{n+1,p}}{\mu_{n+1,p}})=0$ and that  $\|w_p\|_{L^{\infty}(B_{4r}^+(0))}\leq C$.
Then
\[
\|\varphi_p\|_{L^a(B_{3r}(0))}\leq C|B_{3r}(0)|^{\tfrac1a}\leq C\quad\text{for any $p>p_r$ and for any $a>1$.}
\]
Finally by interior elliptic regularity
\begin{equation}\label{picche}
\|\hat\varphi_p\|_{W^{2,q}(B_{2r}(0))}\leq C\left(\|\mu_{n+1,p}^2
(\|w_p\|_{L^\infty(B^+ _{4r}(0))}+1)\|_{L^q(B_{3r}(0))}+\|\hat\varphi_p\|_{L^q(B_{3r}(0))}\right)\leq C.
\end{equation}
Being $s_{n+1,p}=w_p+\|w_p\|_{L^{\infty}(B^+_{4r}(0))}+1-\varphi_p$, combining \eqref{estimationw1} and \eqref{picche} 
we obtain
\[
\|s_{n+1,p}\|_{W^{\frac12+t,q}(B^+_{2r}(0))}\leq C\quad\text{for $q>4$, $0<t<\frac2q$ and $p>p_r$.}
\]
At last by the Morrey embedding theorem we get that
\[
\|s_{n+1,p}\|_{C^{0,\alpha}(B^+_{2r}(0))}\leq C\quad\text{for some $\alpha>0$}
\]
and in turn, by Schauder estimates for the Neumann problem, we get
\[
\|s_{n+1,p}\|_{C^{1,\alpha}(B^+_{r}(0))}\leq C\left(\|-\mu_{n+1,p}^2 p\|_{L^\infty(B^+_{2r}(0))}+\|(1+\frac{s_{n+1,p}}{p})^p\|_{C^{0,\alpha}(D_{2r}(0))}+\|s_{n+1,p}\|_{C^{0,\alpha}(B^+_{2r}(0))}\right)\leq C
\]
for any $p>p_r$, so \eqref{boundC1alpha} holds true.

\

By \eqref{boundC1alpha} and the regularity theory of elliptic equations, we derive that, up to a subsequence, 
\begin{equation}\label{convergence0}
s_{n+1,p}\rightarrow \tilde U \mbox{ in }  C^1_{loc} (\overline{\mathbb R^{2}})
\mbox{ as }p\rightarrow\infty,
\end{equation}
where, by \eqref{systemsn+1,p} and \eqref{**persn}, $\tilde U$ satisfies the
 following problem
\begin{equation}\label{limitproblem2}
\left\{\begin{array}{ll}
\Delta \tilde U= 0&  \mbox{ in }\mathbb R^{2}
\\
\frac{\partial \tilde U}{\partial \nu}=e^{\tilde U}& \mbox{ on }\partial \mathbb R^{2}
\\
\tilde U\leq 0 &  \mbox{ in }\mathbb R^{2}.
\end{array}\right.
\end{equation}
Furthermore $s_{n+1,p}(0, \frac{x\coorddue_{n+1,p}}{\mu_{n+1,p}})=0$, for any $ p$, then by \eqref{riscrivoCaseL} it follows that $\tilde U(0,L)=0$.

Let us now prove that
\begin{equation}\label{boundness1}\int_{\partial \mathbb R^{2}} e^{\tilde U} < \infty.\end{equation}

Let $R>0$, then for any $|t\coorduno|<R$
we have
$$\displaystyle(p+1)\left[\log \left(1+\frac{s_{n+1,p}(t\coorduno,0)}{p}\right)-\frac{s_{n+1,p}(t\coorduno,0)}{p+1}\right]\underset{p\rightarrow+\infty}{\longrightarrow} 0, $$
so we can use Fatou's Lemma and \eqref{convergence0} to write
\begin{eqnarray*}
\displaystyle\int_{-R}^R e^{\tilde U(t\coorduno,0)}dt\coorduno&\leq&
\int_{-R}^{R}e^{s_{n+1,p}(t\coorduno, 0)+(p+1)\left[\log \left(1+\frac{s_{n+1,p}(t\coorduno,0)}{p}\right)-\frac{s_{n+1,p}(t\coorduno,0)}{p+1}\right]}dt\coorduno+o_p(1)\\
&\leq&
\int_{B_{R}(0)\cap\{t\coorddue=0\}}\left(1+\frac{s_{n+1,p}(t)}{p}\right)^{p+1}dt\coorduno+o_p(1)\\
&\leq&\int_{B_{R}(0)\cap\{t\coorddue=0\}}\fr{u_p^{p+1}(\hat x_{n+1,p}+\mu_{n+1,p}t)}{u_p^{p+1}(x_{n+1,p})}dt\coorduno+o_p(1)\\
&\leq&\mu_{n+1,p}^{-1}
\int_{B_{R\mu_{n+1,p}}(\hat x_{n+1,p})\cap \{x\coorddue=0\}}\fr{u^{p+1}_p(x)}{u^{p+1}_p(x_{n+1,p})}dx\coorduno+o_p(1)\\
&\leq&\fr p{u_p(x_{n+1,p})^2}\int_{\partial
\Omega}\upp^{p+1}(x)d\sigma(x)+o_p(1)
\\
&\stackrel{\eqref{geq1}}{\leq}&\fr
p{(1-\ep)^2}\int_{\partial\Omega}\upp^{p+1}(x)
d\sigma(x)+o_p(1)\stackrel{\eqref{energylimit}}{\leq} C<+\infty,
\end{eqnarray*}
so that $e^{\tilde U}\in L^1(\partial\mathbb R^{2})$.

Using now \eqref{limitproblem2}, \eqref{boundness1} and 
 the classification due to P. Liu \cite{Liu} (see also \cite{ZZ}), we obtain
\begin{equation}
\label{ULiuClas}
\tilde U(t\coorduno,t\coorddue)=\log \frac{2\eta_2}{(t\coorduno-\eta_{1})^2+(t\coorddue+\eta_{2})^2} \ \hbox{where } \eta_1\in
\mathbb{R}\ \hbox{and }  \eta_2>0.
\end{equation}

\

\begin{remark}\label{remark:L=0}
Notice that what we have proven up to now in Case $(1)$ holds true also if 
\[dist(x_{n+1,p},\partial \Omega)\mu^{-1}_{n+1,p}\longrightarrow
	L=0,\] in particular we get that the rescaled function $s_{n+1}$ defined in \eqref{cuoris}
	converges  to $\tilde U$ (see \eqref{convergence0}),
which in this case satisfies the conditions $\tilde U(0,0)=0$ and $\tilde U\leq 0$, that imply that $\eta_1=0$ and $\eta_2=2$, namely that $\tilde U\equiv U$, where $U$ is the function defined in \eqref{v0}.
\end{remark}

\

We remark that $\tilde U$ is a radial and decreasing
function with respect to the point $(\eta_{1},-\eta_{2})$. We
have
$$\tilde U(\eta_{1},0)>\tilde U(0,L)=0$$ which
contradicts the fact that $\tilde U\leq 0$.\\
This rules out Case $(1)$, namely the possibility that $dist(x_{n+1,p},\partial \Omega)\mu^{-1}_{n+1,p}\longrightarrow
	L>0$.

\

\underline{Case $(2)$.} In the sequel we will prove that the second case  i.e. $dist(x_{n+1,p},
\partial \Omega)\mu_{n+1,p}^{-1}\longrightarrow +\infty$ can not occur too. Let $v_{n+1,p}$ be the rescaled function defined in \eqref{cuori}.
 By
\eqref{problem1}, $v_{n+1,p}$ solves
$$ -\Delta v_{n+1,p} +\mu_{n+1,p}^2v_{n+1,p}=-\mu_{n+1,p}^2p \quad \hbox{in }\widetilde{\Omega}_{n+1,p}.$$
Since for any $R>0$, $B_R(0)\subset \widetilde{\Omega}_{n+1,p}$ for
$p$ large enough,  it follows that
$\widetilde{\Omega}_{n+1,p}$ converges to the whole plane
$\mathbb{R}^2$.\\ 
Furthermore, from \eqref{geq1}, \eqref{mk+1} and the uniform bound \eqref{boundSoluzio} we get
\begin{equation}\label{Fiori}
|\Delta v_{n+1,p}|\leq |\mu_{n+1,p}^2v_{n+1,p}|+|\mu_{n+1,p}^2p|\leq C \hbox{ in }\widetilde{\Omega}_{n+1,p},
\end{equation}
 namely $v_{n+1,p}$ are functions with uniformly
bounded laplacian in $B_R(0)$, moreover  $v_{n+1,p}(0)=0$.  
\\
Let us decompose
\[v_{n+1,p}=\varphi_{p}+\psi_{p}\quad\text{ in }\widetilde\Omega_{n+1,p}\cap B_{R}(0),\]
with $-\Delta\varphi_{p}=-\Delta v_{n+1,p}$ in $\widetilde\Omega_{n+1,p}\cap B_{R}(0)$
and $\psi_{p}=v_{n+1,p}$ in $\partial (\widetilde\Omega_{n+1,p}\cap B_{R}(0))$.
Using \eqref{Fiori} by standard elliptic theory, we see that $\varphi_{p}$ is uniformly bounded in $\widetilde\Omega_{n+1,p}\cap B_{R}(0)$. The function $\psi_{p}$ is harmonic in $\widetilde\Omega_{n+1,p}\cap B_{R}(0)$ and bounded from above by \eqref{**}. By the Harnack inequality, either $\psi_{p}$
is uniformly bounded in $B_{R}(0)$, or it tends to $-\infty$ on each compact set of $B_{R}(0)$. The second alternative cannot happen because, by definition, $\psi_{p}(0)=v_{n+1,p}(0)-\varphi_{p}(0)=-\varphi_{p}(0)\geq -C$. Hence $v_{n+1,p}$ is uniformly bounded in $B_{R}(0)$ for all $R>0$. After passing to a subsequence, standard elliptic theory implies that $v_{n+1,p}$ is bounded in $C_{loc}^{2}(\mathbb R^{2})$. Thus 
\begin{equation}\label{convergence}
v_{n+1,p}\rightarrow V \mbox{ in }  C^1_{loc} (\mathbb{R}^2) \mbox{
as }p\rightarrow\infty,
\end{equation}
where $V\in C^{1}(\mathbb R^{2})$ is a harmonic function satisfying $V(0)=0$ and $V\leq 0$ by \eqref{**}.

So by a Liouville type theorem 
\begin{equation}\label{VNulla}V\equiv 0.\end{equation}
Let now  $Q_{n+1}=\displaystyle\lim_{p\rightarrow\infty} x_{n+1,p}$.
By \eqref{FirstInclusion} and \eqref{mk+1} it follows that $Q_{n+1}\in\mathcal S\subset \partial \Omega$, where $\mathcal S$ is the concentration set in \eqref{def:blowupset}.
In order to simplify the exposition,
we will assume in the sequel that $\partial \Omega$ is flat near the
point $Q_{n+1}$. This flatness assumption means that there exists
$R_0>0$ such that
$$\Omega \cap B_{R_0}^+(Q_{n+1})=B_{R_0}^+(Q_{n+1})\hbox{ and
}D_{R_0}(Q_{n+1})\subset
\partial \Omega.$$ We also assume that
near $Q_{n+1}$, $\partial \Omega$ is contained in the hyperplane
$x\coorddue=0$ and the unit outward normal to $\partial \Omega$ at
$Q_{n+1}$ is $(-e\coorddue)$, where $e\coorddue$ is the second element of the
canonical basis of $\mathbb{R}^2$.\\
W.l.o.g. we can also assume that 
\begin{equation}\label{maggio}B_{R_0}(Q_{n+1})\cap\mathcal S=\{Q_{n+1}\}.\end{equation}
Now, inspired by Guo-Liu \cite{GL} (see page 750), we define the function
$$W_p(x)=\displaystyle \frac{p}{u_p^2(x_{n+1,p})}\int_{-s_0}^{s_0}u_p^{p+1}(x+(s,0))ds, \quad
\forall x\in\overline{B}^+_{2s_0}(Q_{n+1}),
$$ where $0<s_0<R_0/4$.
We have
 \begin{eqnarray*}
\Delta_x \left(u^{p+1}\left(x+(s,0)\right)\right)&=&
(p+1)u^{p}_{p}\left(x+(s,0)\right)\Delta_x
u_{p}\left(x+(s,0)\right)\\&&+
(p+1)\,p\,u^{p-1}_{p}\left(x+(s,0)\right)\left|\nabla_x u_{p}\left(x+(s,0)\right)\right|^2\\
&\geq&(p+1)\,u^{p}_{p}\left(x+(s,0)\right)\Delta_x
u_{p}\left(x+(s,0)\right)
\\
&=&(p+1)\,u^{p+1}_p\left(x+(s,0)\right)\\&\geq& 0\quad  \forall
x\in B^+_{2s_0}(Q_{n+1})\cup S_{2s_{0}}(Q_{n+1})\hbox{ and }\forall s\in
[-s_0,s_0].
 \end{eqnarray*}
Hence $W_p$ is a subharmonic
 function in $B^+_{2s_0}(Q_{n+1})$.\\
By \eqref{convpupNEWW} and \eqref{maggio}, for $p$ large,
\[
|u_p(y)|\leq C_{1}\frac{1}{p}, \quad \forall\, y\in \overline{S_{2s_0}(Q_{n+1})},
\]
thus
 \begin{equation}\label{Wpvaazero}
 W_p\rightarrow 0\quad\mbox{ uniformly  in }\overline{S_{2s_0}(Q_{n+1})}.
 \end{equation} 
  Furthermore for each $y\in D_{2s_0}$ we have
\begin{equation}\label{wdoppiobounded}
W_p(y)\leq \frac{p}{u_p^{2}(x_{n+1,p})}\int_{\partial
\Omega}u^{p+1}(x)\ d\sigma(x)\stackrel{\eqref{energylimit},\eqref{geq1}}{\leq} C_{2}.
\end{equation}
Combining \eqref{Wpvaazero} with \eqref{wdoppiobounded} and using the maximum principle we get
\begin{equation}\label{wmajor}
W_p(x)\leq C \hbox{ for all } x\in \overline{B}^+_{2s_0}(Q_{n+1}).
\end{equation}
On the other hand, we have for $k> C$ and $p$ sufficiently large
\begin{eqnarray*}
\displaystyle
W_p(x_{n+1,p})&=&\frac{p}{u_p^{2}(x_{n+1,p})}\int_{-s_0}^{s_0}u^{p+1}_p(x_{n+1,p}+(s,0))ds\\
&\stackrel{\eqref{mk+1}}{\geq}&\frac{p}{u_p^{2}(x_{n+1,p})}\int_{-k\mu_{n+1,p}}^{k\mu_{n+1,p}}u^{p+1}_p(x_{n+1,p}+(s,0))ds\\
&\geq&\frac{p \ \mu_{n+1,p}}{u_p^{2}(x_{n+1,p})}\int_{-k}^{k}u^{p+1}_p(x_{n+1,p}+\mu_{n+1,p}(t,0))dt\\
&\geq&\int_{-k}^{k}\left(\frac{u_p(x_{n+1,p}+\mu_{n+1,p}(t,0))}{u_p(x_{n+1,p})}\right)^{p+1}dt\\
&\geq& \int_{-k}^{k}\left(1+\frac{v_{n+1,p}(t,0)}{p}\right)^{p+1}dt\\
&\stackrel{\eqref{convergence}}{=}&\int_{-k}^{k}e^{V(t,0)}dt+o(1)\stackrel{\eqref{VNulla}}{=}2k>2C,
\end{eqnarray*}
which is in contradiction with \eqref{wmajor}.\\
Hence, we have proved that  Case $(2)$ cannot occur.\\
So, up to a subsequence, $$\displaystyle\lim_{p\rightarrow
\infty}dist(x_{n+1,p},
\partial \Omega)\mu^{-1}_{n+1,p}=0$$ and
$(\mathcal{P}_2^{n+1})$ holds with the added points $(x_{n+1,p})$.\\

\textit{\underline{Proof of $(\mathcal{P}_3^{n+1})$}.}\\
Since $(\mathcal{P}_2^{n+1})$ holds, by Remark \ref{remark:L=0}, assuming the flatness of $\partial\Omega$ near $Q_{n+1}$, we have that 
\[s_{n+1,p}\rightarrow U\quad\mbox{ in }C^{1}_{loc}(\overline{\mathbb R^{2}})\]
where  $U$ is the limit function in \eqref{v0}.\\
By the definition of $s_{n+1,p}$ (see \eqref{cuoris}), since  $x\coorddue_{n+1,p} / \mu_{n+1,p}\to0$, we conclude that also 
\[v_{n+1,p}\rightarrow U\quad\mbox{ in }C^{1}_{loc}(\overline{\mathbb R^{2}}),\]
where $v_{n+1,p}$ are the rescaled functions introduced in \eqref{cuori}.

In the flat case this proves that $(\mathcal{P}_3^{n+1})$ holds with the added points
$(x_{n+1,p})$, since the rescaled function $z_{n+1,p }$ in \eqref{vipdef} coincide with the rescaled functions $v_{n+1,p}$.
The non-flat case can be obtained in the same way (see \cite[Theorem 3]{Castro1}), thus \textit{STEP 2.} is proved.

\

\

\textit{STEP 3. We complete the proof of Proposition \ref{thm:x1N}.}

\

From \textit{STEP 1.} we have that $(\mathcal{P}_1^1)$,
$(\mathcal{P}_2^1)$ and $(\mathcal{P}_3^1)$ hold. Then, by \textit{STEP
2.}, either 
$(\mathcal{P}_4^1)$ holds
or
$(\mathcal{P}_1^2)$, $(\mathcal{P}_2^2)$ and
$(\mathcal{P}_3^2)$ hold. In the first
case the assertion is proved with $k=1$. In the second case we go on
and proceed with the same alternative until we reach a number
$k\in\N\setminus\{0\}$ for which $(\mathcal{P}_1^{k})$,
$(\mathcal{P}_2^{k})$, $(\mathcal{P}_3^{k})$ and
$(\mathcal{P}_4^{k})$ hold up to a sequence. Note that this is
possible because the solutions $u_p$ satisfy  \eqref{energylimit}
and Lemma \ref{lemma:BoundEnergiaBassino} holds and hence the
maximal number $k$ of families of points for which
$(\mathcal{P}_1^{k})$, $(\mathcal{P}_2^{k})$, $(\mathcal{P}_3^{k})$
hold must be finite.

\

Moreover, given any other family of points $x_{k+1,p}$, it is
impossible to extract a new sequence from it such that
$(\mathcal{P}_1^{k+1})$, $(\mathcal{P}_2^{k+1})$,
$(\mathcal{P}_3^{k+1})$ and $(\mathcal{P}_4^{k+1})$ hold together
with the points $(x_{i,p})_{i=1,..,k+1}$. Indeed if
$(\mathcal{P}_1^{k+1})$ was verified then
\[\frac{|x_{k+1,p}-\xip|}{\mu_{k+1,p}}\to+\infty\quad \mbox{ as } p\to+\infty,\ \mbox{ for any }i\in\{1,\ldots,k\},\]
but this would contradict $(\mathcal{P}_4^k)$.

\

Finally let us recall that by Remark \ref{boundaryconcpoints} 
\[\widetilde{\mathcal{S}}=\left\{\lim_{p\to+\infty}\xip,\,i=1,\ldots,k\right\} 	\] 
hence \eqref{PrimaCaratterizzazioneSTildeN}   follows from \eqref{PrimaCaratterizzazioneSTilde}.
\end{proof}

\section{Refined analysis}\label{section:positive}

We know that the solutions of \eqref{problem1} satisfy Theorem \ref{provaI} in the Introduction.
In particular,  for a sequence of positive solutions $u_{p_{n}}$,  the blow-up set $\mathcal S$ of $p_{n}u_{p_{n}}$ is a discrete subset of $\partial\Omega$  
\[\mathcal S=\{\bar x_1,\ldots,\bar x_m\}\subset \partial\Omega\] 
(see \eqref{def:blowupset} for the definition of $\mathcal S$ and $(1)$-Theorem \ref{provaI} for its characterization). 
\\ 
Moreover we have seen in Proposition \ref{thm:x1N} that, up to reordering the points $\bar x_{i}$,  the concentration set $\widetilde{\mathcal{S}}$ of $u_{p_{n}}$, defined in \eqref{S},  satisfies
\[\widetilde{\mathcal{S}}=\{\bar x_1,\ldots ,\bar x_N\}, \quad\mbox{ with }N\leq\min\{m,k\},\]   
where $k$ is the maximal number of bubbles  $U$ in Proposition  \ref{thm:x1N}.\\\\

Thanks to the local analysis developed in the previous section, we can actually deduce the following.
\begin{proposition}
\label{consequence}
$$\mathcal S=\widetilde{\mathcal{S}}\qquad \text{and so in particular\quad $m=N\leq k$.} $$
\end{proposition}
\begin{proof}
It is enough to show that $\mathcal{S}\subseteq \widetilde{\mathcal S}$, namely that $m\leq N$. Let us assume by contradiction that $\bar x_{m}\in\mathcal{S}\setminus    \widetilde{\mathcal{ S}}$. Since $\bar x_{m}\in\mathcal{S}$, by definition  there exists $x_{n}\in\overline{\Omega}$, $x_{n}\rightarrow \bar x_{m}$ such that $p_{n}u_{p_{n}}(x_{n})\rightarrow +\infty$.
Since $\bar x_{m}\not\in\widetilde{\mathcal{ S}}$ then there exist $r>0$ such that $\overline{B_r(\bar x_m)}\cap \widetilde{\mathcal{ S}}=\emptyset$ and that definitively $x_n\in \overline{B_r(\bar x_m)}\cap \overline\Omega=:K$.
Next we show that $(\mathcal{P}^k_4)$, which holds by  Proposition \ref{thm:x1N}, implies that 
\begin{equation}\label{E4}\max_{K}p_{n}|u_{p_{n}}|\leq C, \end{equation} 
thus reaching a contradiction.
\\
Let $\delta:=dist (K,\widetilde{\mathcal S})/2>0$. For $y \in K$ we have
\begin{equation}\label{E1}u_{p_n}(y)=\int_{\partial
\Omega} G(x,y)\frac{\partial u_{p_n}}{\partial \nu} \ d\sigma(x)
=\int_{\partial \Omega} G(x,y)u^{p_n}_{p_n}(x) \
d\sigma(x),\end{equation} where $G(.,.)$ is the Green function
satisfying \eqref{Greenequation}. We split the integral over
$\partial \Omega $ into two parts, the integral over $B_{\delta}(y)\cap
\partial \Omega$ and the integral over $\partial \Omega \setminus
B_{\delta}(y)$.\\
On the one hand, if $x\in B_{\delta}(y)$ we get $d(x,\widetilde{\mathcal{S}})\geq \delta >0$ and so 
$R_{k,{p_n}}(x)\geq c>0$ for $n$ large. From $(\mathcal{P}^k_4)$ we derive
$$u^{{p_n}-1}_{p_n}(x)\leq \frac{C}{{p_n}}, \quad \forall x\in B_{\delta}(y).$$
 Hence
\begin{equation}\label{E2}\int_{B_{\delta}(y)\cap
\partial \Omega} G(x,y) u^{p_n}_{p_n}(x) \ d\sigma(x)\leq
C\left(\frac{1}{p_n}\right)^\frac{{p_n}}{{p_n}-1}\int_{\partial \Omega} G(x,y) \
d\sigma(x)\leq \frac{C}{{p_n}}
\end{equation}
where we have used the fact that $G(.,y)$ is integrable on $\partial \Omega$
which follows from Lemma \ref{Gponctuelestimate} in the Appendix.\\ On the other hand if $x\in
\partial \Omega \setminus B_{\delta}(y)$ then $G(x,y)\leq C$ by using \eqref{Gboundness} since
$|x-y|>\delta>0$. Thus we get
\begin{equation}\label{E3}\int_{\partial\Omega \setminus B_{\delta}(y)} G(x,y) u^{{p_n}}_{p_n}(x) dx\leq
c_{K}\int_{\partial \Omega}u^{p_n}_{p_n}(x) dx\overset{\eqref{boundEnergiap} }{\leq} \frac{c_{K}}{{p_n}}. \end{equation}
Combining \eqref{E1} with \eqref{E2} and \eqref{E3} we deduce \eqref{E4}.
\end{proof}

We conclude the subsection with a result which will be of help in computing the constants $c_i$ which appear in \eqref{convpup} and \eqref{impx_j relazioneintro}.
\begin{lemma}\label{defgammaj}
Let 
$p_n\rightarrow +\infty$, as $n\rightarrow \infty$ and $c_i>0$ be as in Theorem \ref{provaI}, then 
	\[c_i=\lim_{\delta\rightarrow 0}\lim_{n\rightarrow \infty}p_n\int_{B_{\delta}(\bar x_i)\cap\partial\Omega}u_{p_n}^{p_n}\, d\sigma(x)\ , \quad i=1,\ldots, m.\]
	\end{lemma}
\begin{proof}
We retrace the proof of \eqref{convpupNEWW} in order to characterize the constants $c_{i}$.
Let us observe that, since the points $\bar x_i$'s are isolated, there exists $\delta>0$ such that
$B_{2\delta}(\bar x_i)\cap B_{2\delta}(\bar x_j)=\emptyset$. Then by
the Green representation formula for each
$x\in\overline{\Omega}\setminus {\mathcal{S}}$ we have
\begin{eqnarray*}
p_nu_{p_n}(x)&=& p\int_{\partial\Omega}G(x,y)u^{p_n}_{p_n}(y)\,d\sigma(y)\\
&=&
p_n\sum_{i=1}^m\int_{B_{\delta}(\bar x_i)\cap\partial\Omega}G(x,y)u^{p_n}_{p_n}(y)\,d\sigma(y)+
\,p_n\int_{\partial\Omega\setminus\cup_{i} B_{\delta}(\bar x_i)}G(x,y)u^{p_n}_{p_n}(y)\,d\sigma(y)\\
&=&p_n\sum_{i=1}^m\int_{B_{\delta}(\bar x_i)\cap\partial\Omega}G(x,y)u^{p_n}_{p_n}(y)\,d\sigma(y)
+ o_n(1),
\end{eqnarray*}
where in the last equality we have used that $p_nu_{p_n}$ is bounded on compact sets of $\overline\Omega\setminus\mathcal S$ and the fact
that $\int_{\partial \Omega}G(x,y)\ d\sigma(y)\leq C$ (from Lemma \ref{Gponctuelestimate} in the Appendix).\\
Furthermore by the continuity of $G(x,\cdot)$ in $\overline\Omega
\setminus\{x\}$ and by 
\eqref{boundEnergiap} we obtain, up to a subsequence, that
\[
\lim_{n\to+\infty}p_nu_{p_n}(x)=\sum_{i=1}^m c_i G(x,\bar x_i),\qquad\text{where}\quad c_i:=\lim_{\delta\rightarrow 0}\lim_{n\rightarrow \infty}p_n\int_{B_{\delta}(\bar x_i)\cap\partial\Omega}u_{p_n}^{p_n}\, d\sigma(x).\]
\end{proof}
\

\subsection{Scaling around local maxima} \label{SectionScaling around local maxima}$\;$\\

Let  $x_{i,p_n}\in\overline\Omega$,  $i=1,\ldots, k$, be  the maximal number of concentrating sequences in Proposition  \ref{thm:x1N}. W.l.o.g. we can relabel them
 and assume for the first $m$ sequences  that
\begin{equation}\label{ribattezzo}
x_{j,p_n}\rightarrow \bar x_j,\ \forall j=1,\ldots, m\quad\mbox{ and
}\quad{\mathcal{S}}=\{\bar x_1, \ \bar x_2, \ \ldots, \bar x_m \}
\end{equation}

 In order to simplify the exposition, we will
assume in the sequel that $\partial \Omega$ is flat near $\bar x_j\hbox{
for all } j= 1,\ldots ,m$. This flatness assumption means that there
exists $R_{j}>0$ such that
\begin{equation}
\label{hpFlat}
\Omega \cap B_{R_{j}}^+ (\bar x_j)=B_{R_{j}}^+(\bar x_j)\hbox{ and
}D_{R_{j}}(\bar x_j)\subset
\partial \Omega, \hbox{ for all } j= 1,\ldots ,m,
\end{equation} 
and moreover $\Psi_j\equiv\Id$.
Since the $\bar x_j$'s are distinct, it follows that there exists $r\in(0,\min_{j=1,\ldots, m} R_{j}/4)$ such that
\begin{equation}\label{rPiccoloAbbastanza}
B_{4r}^+(\bar x_\ell) \cap B_{4r}^+(\bar x_j)=\emptyset,\
B_{4r}^+(\bar x_j)\subset\Omega, \hbox{ for all } \ell,j = 1,\ldots ,m, \ell
\neq j.
\end{equation}

\begin{lemma}\label{lemma:possoRiscalareAttornoAiMax}
Let $m \in\mathbb{N}\setminus\{0\}$ be as in \eqref{ribattezzo} and
let $r > 0$ be as in \eqref{rPiccoloAbbastanza}. Let us define
${{y}}_{j,n} \in \overline{B_{2r}^+(\bar x_j)}, j = 1, \ldots,m$, such
that
\begin{equation}\label{4.2}
u_{p_{n}}({{y}}_{j,n}):=\displaystyle
\max_{\overline{B_{2r}^+(\bar x_j)}}u_{p_{n}}(x).
\end{equation}
Then, for any $j = 1,\ldots,m$ and as $n \rightarrow \infty$:
\begin{itemize}
\item[$(i)$] \begin{equation}\label{(2.19)}
\varepsilon_{j,n}:=
\left[p_{n}u_{p_{n}}^{p_{n}-1}({{y}}_{j,n})\right]^{-1}\longrightarrow 0.
\end{equation}
\item[$(ii)$] \begin{equation}\label{(2.17)}
{{y}}_{j,n}\longrightarrow  \bar  x_j\quad \hbox{and} \quad
{{y}}_{j,n}\in\partial \Omega \hbox{ for }n\hbox{
large.}\end{equation}
\item[$(iii)$] $$\frac{|{{y}}_{i,n} -{{y}}_{j,n}|}{\varepsilon_{j,n}}
\longrightarrow +\infty \hbox{ for any } i = 1, \ldots ,m, i \neq
j.$$
\item[$(iv)$] Defining:
\begin{equation}\label{defRiscalataMax}
w_{j,n}(y):= \frac{p_{n}} {u_{p_{n}}({{y}}_{j,n})} (u_{p_{n}}({{y}}_{j,n} +
\varepsilon_{j,n}y)- u_{p_{n}}({{y}}_{j,n})), \quad y \in\Omega_{j, n}
:=\frac{\Omega-{{y}}_{j,n}}{\varepsilon_{j,n}},
\end{equation}
then
\begin{equation}\label{(2.21)}
\lim_{n\rightarrow \infty}w_{j,n}= U \ \mbox{ in }C^1_{loc}(\mathbb{R}^2_+)
\end{equation}
with $U$ as in \eqref{v0}.
\item[$(v)$] 
$$\displaystyle \liminf_{n\rightarrow \infty} \frac{p_{n}}{u_{p_{n}}({{y}}_{j,n})}\int_{D_r({{y}}_{j,n} )} u_{p_{n}}^{p_{n}}(x)\, d\sigma(x) \geq
2\pi .$$
\end{itemize}
\end{lemma}

\begin{remark}\label{rmkinfsup} 
$(i)$ is the analogue of \eqref{muVaAZero}-\eqref{mip} for the families of points ${{y}}_{j,n}$, $j=1,\ldots, m$.\\
$(iii)$ and $(iv)$ are respectively properties  $(\mathcal{P}_1^m)$ and $(\mathcal{P}_3^m)$ introduced in Section \ref{sectionGeneralAsy}.
Moreover by $(i)$ we get
\begin{equation}
\label{maxLocMag1} \liminf_{n\rightarrow \infty}
u_{p_{n}}({{y}}_{j,n})\geq 1
\end{equation}
 and by (ii) we also deduce property $(\mathcal{P}_2^m)$ and  that
for any $\delta\in(0,2r)$ there exists $n_{\delta}\in\mathbb N$ such that
\begin{equation}\label{massimoVaAfinireInPallaPiccolaApiacere}
{{y}}_{j,n}\in D_{\delta} (\bar x_j),\quad\mbox{ for } n\geq
n_{\delta}.
\end{equation}
\end{remark}

\

\begin{proof}

$(i)$: let $\bar x_j\in\mathcal{S}=\widetilde{\mathcal{S}} $ by Proposition \ref{consequence} then $\exists$ a sequence 
$x_{n}\rightarrow \bar x_j$ such that $p_{n}u_{p_{n}}^{p_{n}-1}(x_{n})\rightarrow +\infty$ as $n\rightarrow \infty$.
Hence $x_{n}\in
B_r(\bar x_j)$ for $n$ large and the assertion follows observing that
by definition $u_{p_{n}}({{y}}_{j,n})\geq u_{p_{n}}(x_{n})$.

\

$(ii)$: Assume by contradiction that ${{y}}_{j,n}$ does not converge to $\bar x_j$, then up to a subsequence (that we still denote by ${{y}}_{j,n}$) ${{y}}_{j,n}\rightarrow \widetilde x$ such that $(2r\geq)\ |\bar x_j-\widetilde x|\geq\delta>0$. But then by \eqref{convpup} in Theorem \ref{provaI}
\begin{equation}\label{new}p_{n}u_{p_{n}}({{y}}_{j,n})=\sum^m_{j=1} c_jG(\widetilde x , \bar x_j)+ o_{n}(1)=O(1).\end{equation}
Moreover, since $\bar x_{j}\in \mathcal S$, there exists  a sequence $x_{n}\in \overline\Omega$ such that $x_{n}\rightarrow \bar x_{j}$ and $p_{n}u_{p_{n}}(x_{n})\rightarrow +\infty$. Hence $x_{n}\in
B_r(\bar x_j)$ for $n$ large and  
by definition $u_{p_{n}}({{y}}_{j,n})\geq u_{p_{n}}(x_{n})$, which is in contradiction with \eqref{new}, as a consequence ${{y}}_{j,n}\rightarrow
\bar x_j$.

\

Recall that ${{y}}_{j,n}$ satisfies \eqref{4.2} and $\Delta u_{p_{n}}=
u_{p_{n}}$ in $B_{2r}^+(\bar x_j)$. If by contradiction ${{y}}_{j,n}\in B_{2r}^+(\bar x_j)$, then $\Delta u_{p_{n}}({{y}}_{j,n})>0$, which is impossible. Hence
${{y}}_{j,n}\in
\partial B_{2r}^+(\bar x_j)= D_{2r}(\bar x_j)\cup S_{2r}(\bar x_j)$. Since ${{y}}_{j,n}\longrightarrow
\bar x_j\in\partial \Omega$, we obtain ${{y}}_{j,n}\in
 D_{2r}(\bar x_j)\subset \partial \Omega$ for $n$ large.

\

$(iii)$: Just observing that by construction
$|{{y}}_{i,p}-{{y}}_{j,p}|\geq 4r$ if $i\neq j$.

\

$(iv)$: Observe that $(ii)$ and $(i)$ imply that for any $R>0$ there exists   $n_R\in\mathbb N$
such that
\begin{equation}\label{inclusioniSets}
B_R^+(0)\subset
B_{\frac{2r}{\varepsilon_{j,n}}}^+\left(\frac{\bar x_j-{{y}}_{j,p}}{\varepsilon_{j,n}}
\right)\subset \Omega_{j,n}\ \mbox{
 for }\ n\geq n_R.
 \end{equation}
Indeed  for $n$
large
${{y}}_{j,n}\in D_{r}(\bar x_j)$ by $(ii)$ and 
and  $R\varepsilon_{j,n}<r$ by $(i)$.  As a consequence $
B_{R\varepsilon_{j,n}}^+({{y}}_{j,n})\subset
B_{2r}^+(\bar x_j)\subset\Omega$ for $n$ large, which gives
\eqref{inclusioniSets} by scaling back.

\

From \eqref{inclusioniSets} and the arbitrariness of $R$
we deduce that the set $\Omega_{j,n}
  \rightarrow \mathbb R^2_+$
as $n\rightarrow \infty$.

\eqref{(2.21)} is then obtained similarly as in the proof of Proposition \ref{thm:x1N}-$(\mathcal{P}_3^{n+1})$.
\\
\\
$(v)$: using
\eqref{massimoVaAfinireInPallaPiccolaApiacere} we have that
${{y}}_{j,n}\in D_{r}(\bar x_j)$ for large $n$ and  so
$B_{r}^+({{y}}_{j,n})\subset B_{2r}^+(\bar x_j)\subset\Omega$
for $n$ large, namely, by scaling
\begin{equation}
\label{inclusioniSets2} 
B_{ \frac{r}{ \varepsilon_{j,n}
}}^+(0)\subset \Omega_{j,n},\quad \mbox{ for $n$ large }
\end{equation}
By scaling and passing to the limit as $n\rightarrow \infty$, by $(i)$, $(iv)$ and
Fatou's Lemma one has 
\begin{eqnarray*}
\liminf_{n\rightarrow \infty} \frac{p_{n}}{u_{p_{n}}({{y}}_{j,n})}\int_{D_{r}({{y}}_{j,n})}u^{p_{n}}_{p_{n}}(x)\,
d\sigma(x)
&=&
\liminf_{n\rightarrow \infty} \int_{D_{\frac{r}{\varepsilon_{j,n}}}(0)}
\left(1+\frac{w_{j,n}(y)}{p_{n}}\right)^{p_{n}}\, d\sigma(y)
\\
&\geq&
 \int_{\partial\R^2_+}e^{U(y)}\,
d\sigma(y)= 2\pi
\end{eqnarray*}
which gives $(v)$.
\end{proof}

\

\begin{lemma}\label{propbeta}
Let $r > 0$ be as in \eqref{rPiccoloAbbastanza} and 
${{y}}_{j,n} \hbox{ for } j = 1,\ldots ,m$ be the local maxima of
$u_{p_{n}}$ as in \eqref{4.2}.
Let us define 
\begin{equation}\label{betap}
 \beta_{j,n}:=\displaystyle
\frac{p_{n}}{ u_{p_{n}}({{y}}_{j,n})} \int_{ D_r({{y}}_{j,n})} u_{p_{n}}^{p_{n}}(x)
\,d\sigma(x), \quad  \mbox{ for }j =
1,\ldots,m.
\end{equation}
Then \begin{equation}
 \displaystyle\lim_{n\longrightarrow +\infty} \beta_{j,n}=2\pi.
\end{equation}
\end{lemma}
\begin{proof}
Fix $j\in\{1,\ldots,m\}.$ By Lemma
\ref{lemma:possoRiscalareAttornoAiMax}-$(v)$ we already know  that
\[
\liminf_{n\rightarrow \infty}\beta_{j,n}\geq 2\pi,
\]
so we have to prove only the opposite inequality:
\begin{equation}\label{betaMin}
\limsup_{n\rightarrow \infty}\beta_{j,n}\leq 2\pi.
\end{equation}
 For $\delta\in (0,r)$ by \eqref{rPiccoloAbbastanza}
 \begin{equation}\label{pallettinaDentroIns}
 B_{\delta}^+(\bar x_j)\subset\Omega
 \end{equation} and we define
 \begin{equation}\label{defalphaj}
 \alpha_{j,n}(\delta):=\frac{p_{n}}{u_{p_{n}}({{y}}_{j,n})}\int_{D_{\delta}(\bar x_j)}u_{p_{n}}^{p_{n}}(x)\,d\sigma(x).
 \end{equation}
 In order to prove \eqref{betaMin} it is sufficient to show that
\begin{equation}\label{alphadelta}
\lim_{\delta\rightarrow 0}\limsup_{n\rightarrow
\infty}\alpha_{j,n}(\delta)\leq 2\pi
\end{equation}
since \eqref{betaMin} will follow observing that
\begin{equation}\label{betaEAlfaHannoStessoLimite}
\beta_{j,n}= \alpha_{j,n}(\delta) +
\frac{p_{n}}{u_{p_{n}}({{y}}_{j,p_n})}\int_{D_{r}({{y}}_{j,n})\setminus
D_{\delta}(\bar x_j)}u_{p_{n}}^{p_{n}}(x)\, d\sigma(x) =  \alpha_{j,n}(\delta) +
o_n(1),
\end{equation}
where the second term goes to zero as $n\rightarrow \infty$ because
${{y}}_{j,n}\in\overline{D_{2r}(\bar x_j)}$. Indeed
$D_{r}({{y}}_{j,n})\setminus D_{\delta}(\bar x_j)\subset
D_{3r}(\bar x_j)\setminus D_{\delta}(\bar x_j)\subset\partial\Omega\setminus
{\mathcal S}$ and we know that for any compact subset of
$\partial\Omega\setminus {\mathcal S}$ the limit \eqref{convpup} holds
and $\liminf_{n\rightarrow \infty} u_{p_{n}}({{y}}_{j,n})\geq 1$ by \eqref{maxLocMag1}.
\\

In the rest of the proof we show \eqref{alphadelta}.\\
Let us consider the 	local Pohozaev identity for problem \eqref{problem1} in the
set $B_{\delta}^+(\bar x_i)$:
\begin{equation}\label{pohozaevidentity1} \displaystyle\int_{B_{\delta}^+(\bar x_i)}u_{p_{n}}^2dx=\int_{\partial
B_{\delta}^+(\bar x_i)} \frac{1}{2}\langle x-\bar x_i,\nu\rangle(|\nabla u_{p_{n}}|^2+u_{p_{n}}^2)
-\langle x-\bar x_i,\nabla u_{p_{n}}\rangle\frac{\partial u_{p_{n}}}{\partial \nu}\ d\sigma(x),
\end{equation}
where $\nu$ is the outer unitary normal vector to $\partial
B_{\delta}^+(\bar x_i)$ in $x$. Recalling that we have assumed that $\partial \Omega$ is flat
near $\bar x_i$ (see \eqref{hpFlat}), then we have $\nu=-e\coorddue$ on $D_\delta(\bar x_i)$, so that
$\langle x-\bar x_i,\nu\rangle=0$ and $\langle x-\bar x_i,\nabla u\rangle=(x-\bar x_i)\coorduno\frac{\partial
u}{\partial x\coorduno}$ for each  $x\in D_\delta(\bar x_i)$. Furthermore on $S_\delta(\bar x_i)$ we have
$\nu=\frac{x-\bar x_i}{\delta}$ and $\langle x-\bar x_i,\nabla u\rangle
=\delta\frac{\partial u}{\partial \nu}$. Hence from 
\eqref{pohozaevidentity1} and by integrating by part we get
\begin{eqnarray*}\frac{1}{p_{n}+1}\displaystyle\int_{D_{\delta}(\bar x_i)}u_{p_{n}}^{p_{n}+1}\,d\sigma &=&
\int_{B_{\delta}^+(\bar x_i)}u_{p_{n}}^2dx
-\frac{\delta}{2}\int_{S_\delta(\bar x_i)} \left(|\nabla u_{p_{n}}|^2+u_{p_{n}}^2
-2\left(\frac{\partial u_{p_{n}}}{\partial \nu}\right)^2\right)\,d\sigma\\
&&+\left[(x-\bar x_i)\coorduno\frac{u_{p_{n}}^{p_{n}+1}}{p_{n}+1}\right]^{\bar x_i\coorduno+\delta}_{\bar x_i\coorduno-\delta}.
\end{eqnarray*}
Multiplying the
last equation by $p_{n}^2$ we obtain
\begin{eqnarray}\label{quaQua}
\frac{p_{n}^2}{p_{n}+1}\int_{ D_{\delta}(\bar x_i)}u_{p_{n}}^{p_{n}+1}\,d\sigma &
\geq & -\frac{\delta}{2} \int_{S_{\delta}(\bar x_i)} | p_{n}\nabla u_{p_{n}}|^2
\,d\sigma -\frac{\delta}{2}\int_{S_{\delta}(\bar x_i)} ( p_{n} u_{p_{n}}
)^2\,d\sigma \nonumber\\&&  +\ \delta \int_{S_{\delta}(\bar x_i)}
\left(p_{n}\frac{\partial u_{p_{n}}}{\partial \nu}\right)^2 \,d\sigma
 + \ p_{n}^2
\left[(x-\bar x_i)\coorduno\frac{u_{p_{n}}^{p_{n}+1}}{p_{n}+1}\right]^{\bar x_i\coorduno+\delta}_{\bar x_i\coorduno-\delta}\nonumber\\
&=:& T_1+T_2+T_3+T_4.
\end{eqnarray}
Next we analyze the behavior of the four terms in the right hand
side.

\

Recall that, by \eqref{convpupNEWW}, $p_{n}u_{p_{n}}\rightarrow
\sum_{j=1}^mc_j G(\cdot,\bar x_j)$ in
$C^1_{loc}(\overline{B_{r}^+(\bar x_i)}\setminus\{\bar x_i\})$. Moreover,
using Lemma \ref{Greenestimate}, for  $\delta\in (0,r)$ we have
\begin{equation}\label{sommG}
\sum_{j=1}^mc_j G(x,\bar x_j)=\frac{c_i}{\pi}\log \frac{1}{
|x - \bar x_i|} + O(1)\ ,\qquad \sum_{j=1}^mc_j\nabla G(x,\bar x_j)= \
-\frac{c_i}{\pi}\frac{x-\bar x_i}{|x-\bar x_i|^2} + O(1)
\end{equation}
for each $x\in \overline{B_{\delta}^+(\bar x_i)}\setminus\{\bar x_i\}$.
By the uniform convergence of the derivative of  $p_{n}u_{p_{n}}$  on compact
sets combined with \eqref{sommG}, passing to the limit we have
\[T_1=-\frac{\delta}{2} \int_{S_{\delta}(\bar x_i)}
| p_{n}\nabla u_{p_{n}}  |^2   \,d\sigma  \underset{n\rightarrow
\infty}{\longrightarrow}
-\frac{\delta}{2}\int_{S_{\delta}(\bar x_i)}\left(-\frac{c_i}{\pi}\frac{
x-\bar x_i }{|x-\bar x_i|^2} + O(1)\right)^2 d\sigma(x)
=-\frac{c_i^2}{2\pi} + O(\delta)
\]
\[T_2=-\frac{\delta}{2} \int_{S_{\delta}(\bar x_i)}
( p_{n} u_{p_{n}} )^2  \,d\sigma  \underset{n\rightarrow
\infty}{\longrightarrow}
-\frac{\delta}{2}\int_{S_{\delta}(\bar x_i)}\left(\frac{c_i}{\pi}\log\frac{
1 }{|x-\bar x_i|} + O(1)\right)^2 \,d\sigma(x)  = O(\delta^2\log^2\delta)
\]
\[T_3=\delta  \int_{S_{\delta}(\bar x_i)}
\left(p_{n}\frac{\partial u_{p_{n}}}{\partial \nu}\right)^2 \,d\sigma
\underset{n\rightarrow \infty}{\longrightarrow}
\delta\int_{S_{\delta}(\bar x_i)}\left(-\frac{c_i}{\pi}\frac{\langle
x-\bar x_i,\nu(x)\rangle }{|x-\bar x_i|^2} + O(1)\right)^2 \,d\sigma(x)
=\frac{c_i^2}{\pi} + O(\delta)\] and also
\begin{eqnarray*}
&&T_4 \leq\frac{ 2p_{n}}{p_{n}+1} \delta  \|p_{n}
u_{p_{n}}^{p_{n}+1}\|_{L^{\infty}(\partial
D_{\delta}(\bar x_i))}\overset{\eqref{convpupNEWW}}{=}
o_n(1)O({\delta}).
\end{eqnarray*}
So by \eqref{quaQua} and recalling  the definition of $\alpha_{i,n}$
\begin{eqnarray}\label{firsto} 
\alpha_{i,n}(\delta)u_{p_{n}}({{y}}_{i,n})^2
&\overset{\eqref{defalphaj}}{=}& u_{p_{n}}({{y}}_{i,n})\ p_{n} \int_{
D_{\delta}(\bar x_i)}u_{p_{n}}^{p_{n}}(x)\,d\sigma(x) \geq p_n\int_{
D_{\delta}(\bar x_i)}u_{p_{n}}^{p_{n}+1}(x)\,d\sigma(x)\nonumber\\&\overset{\eqref{quaQua}}{\geq}&\frac{c_i^2}{2\pi}
+ O(\delta)+ o_n(1),
\end{eqnarray}
but by Lemma  \ref{defgammaj}, 
\eqref{pallettinaDentroIns} and \eqref{defalphaj} 
\begin{equation}
\label{secondt} c_i=\lim_{\delta\rightarrow 0}\lim_{n\rightarrow \infty} p_{n}\int_{ D_{\delta}(\bar x_i)} u_{p_{n}}^{p_{n}}\,d\sigma(x)
=\lim_{\delta\rightarrow
0}
\lim_{n\rightarrow \infty} \alpha_{i,n}(\delta)u_{p_{n}}({{y}}_{i,n}).
\end{equation}
Combining \eqref{firsto} and \eqref{secondt} we get
\eqref{alphadelta}.
\end{proof}

\

Lemma \ref{propbeta} immediately implies the following result.

\begin{proposition}\label{proposizione:quasiQuanteConclusione}
Let $r > 0$ be as in \eqref{rPiccoloAbbastanza} and let
${{y}}_{j,n}$, for  $j = 1,\ldots ,m$, be the local maxima of
$u_{p_{n}}$ as in \eqref{4.2}, where $m$ is the number of points in the concentration set $\mathcal S$.
Let us consider a subsequence of $p_{n}$ (which we still denote by $p_{n}$) such that
\begin{equation}\label{mSec}
m_j:=\displaystyle\lim_{n\rightarrow \infty}
u_{p_{n}}({{y}}_{j,n})=\lim_{r\rightarrow 0}\lim_{n\rightarrow \infty}
\|u_{p_{n}}\|_{L^{\infty}(\overline{B_{2r}(\bar x_j)\cap \Omega})}
\end{equation}
is well defined for $j=1,\ldots, m$.
Then one has \begin{equation}\label{valoreGamma}
 c_j=2\pi \cdot m_j;
\end{equation}
\begin{equation}\label{energiaSemiQuantizzata}
\lim_{n\rightarrow\infty}p_{n}\int_{\Omega}\left(|\nabla u_{p_{n}}|^2+u_{p_{n}}^2 \right)\ dx
=2\pi\sum ^m_{j=1}m_j^2;
\end{equation}
and 
\begin{equation}\label{N=k}
m=k,
\end{equation}
where $c_j$'s are the constant in Theorem \ref{provaI} and $k\in\mathbb N\setminus\{0\}$ is the maximal number of bubbles given by  Proposition \ref{thm:x1N}.
\end{proposition}
\begin{proof}
Observe that, by \eqref{boundSoluzio}, $m_{j}$ is well defined for a suitable subsequence of $p_{n}$ and, furthermore  $1\leq m_j < \infty$ for any $j = 1,\ldots, m$, by \eqref{maxLocMag1}.

\

 \eqref{valoreGamma} follows  from  some argument already used in the proof of Lemma \ref{propbeta} (see \eqref{defalphaj}, \eqref{betaEAlfaHannoStessoLimite}), indeed we have
\[c_j=\lim_{\delta\rightarrow 0}\lim_{n\rightarrow \infty}p_{n}\int_{D_{\delta}(\bar x_j)}
u_{p_{n}}(x)^{p_{n}}\,d\sigma(x) \overset{\eqref{defalphaj}}{=}\lim_{\delta\rightarrow 0}\lim_{n\rightarrow
\infty}  \alpha_{j,n}
(\delta)u_{p_{n}}(y_{j,n})\overset{\eqref{betaEAlfaHannoStessoLimite}}{=}
 \lim_{n\rightarrow \infty}\beta_{j,n}u_{p_{n}}(y_{j,n})=2\pi\cdot m_j,\]
where the last equality follows from Lemma \ref{propbeta}.

\

Next we prove \eqref{energiaSemiQuantizzata}. Observe that
\begin{eqnarray}\label{centerC}
p_{n}\int_{\Omega}\left(|\nabla u_{p_{n}}|^2 +u_{p_{n}}\right)^2\,dx&=&
p_{n}\int_{\partial\Omega}u_{p_{n}}^{p_{n}+1} \,d\sigma\nonumber\\&=&
\sum_{j=1}^m p_{n}\int_{D_r(\bar x_j)}u_{p_{n}}^{p_{n}+1}\,d\sigma
+p_{n}\int_{\partial\Omega\setminus \cup_{j=1}^m D_r(\bar x_j)}
u_{p_{n}}^{p_{n}+1}\,d\sigma
\nonumber\\
&\overset{\eqref{convpupNEWW}}{=} &
 \sum_{j=1}^m p_{n}\int_{D_r(\bar x_j)}u_{p_{n}}^{p_{n}+1}\,d\sigma + o_n(1).
\end{eqnarray}
Moreover
\begin{equation}\label{perdidi}
p_{n}\int_{D_r(\bar x_j)}u_{p_{n}}^{p_{n}+1}\,d\sigma= p_{n}\int_{D_{\frac{r}{2}}({y}_{j,n})}u_{p_{n}}^{p_{n}+1}\,d\sigma + o_n(1),
 \end{equation}
since  for $n$ large enough $D_{\frac{r}{3}}(\bar x_j)\subset
D_{\frac{r}{2}}({ y}_{j,n})\subset D_r(\bar x_j)$ so that
\[
p_{n}\int_{D_{r}(\bar x_j)\setminus {D_{ \frac{r}{2}}({ y}_{j,n})}
}u_{p_{n}}^{p_{n}+1}\,d\sigma \leq p_{n}\int_{\{x\in D_{r}(\bar x_j), \
\frac{r}{3}<|x-\bar x_j|<r\}}u_{p_{n}}^{p_{n}+1}\,d\sigma
\overset{\eqref{convpupNEWW} }{=} o_p(1).
\]
Let us consider the  remaining term in the right hand side of
\eqref{perdidi} and prove that  
\begin{equation}\label{limenersupinf}
\lim_{n\rightarrow \infty} p_{n}\int_{D_{\frac{r}{2}}({ y}_{j,n})}u_{p_{n}}^{p_{n}+1}\,d\sigma= 2\pi\cdot m_{j}^{2}.
\end{equation}
On the one side, since the families of points ${y}_{j,n}$, $j=1,\ldots,m$, satisfy properties $(\mathcal{P}_1^m)$, $(\mathcal{P}_2^m)$ and $(\mathcal{P}_3^m)$  introduced in Section \ref{sectionGeneralAsy} (see Remark \ref{rmkinfsup}), similarly as in the proof of Lemma \ref{lemma:BoundEnergiaBassino}, using \eqref{mSec}, we obtain that 
\begin{equation}\label{liminfener}\liminf_{{n\rightarrow \infty}}
p_{n}\int_{D_{\frac{r}{2}}({y}_{j,n})}u_{p_{n}}^{p_{n}+1}\,d\sigma
\geq 2\pi\cdot m_{j}^{2}.\end{equation}
On the other side, since $B_{\frac{r}{2}}^+({{y}}_{j,n})\subset B_{2r}^+(\bar x_j)\subset\Omega$
for $n$ large,
\[p_{n}\int_{D_{\frac{r}{2}}({y}_{j,n})}u_{p_{n}}^{p_{n}+1}\,d\sigma\leq u_{p_{n}}({y}_{j,n})\, p_{n}
\int_{D_{r}({y}_{j,n})}u_{p_{n}}^{p_{n}}\,d\sigma\overset{\eqref{betap}}{=} u_{p_{n}}({y}_{j,n})^{2} \beta_{j,n},
\]
so  Lemma \ref{propbeta} implies that
\begin{equation}\label{limsupener}\limsup_{n\rightarrow\infty}p_{n}\int_{D_{\frac{r}{2}}({y}_{j,n})}u_{p_{n}}^{p_{n}+1}\,d\sigma\leq 2\pi\cdot m_{j}^{2}.\end{equation}
\eqref{limenersupinf} then follows by combining \eqref{liminfener} and \eqref{limsupener}.
Finally \eqref{centerC}, \eqref{perdidi} and \eqref{limenersupinf} imply \eqref{energiaSemiQuantizzata}.

\

Next we show that the points ${{y}}_{j,{n}}, \ j
=1,\ldots,m$, also satisfy property $(\mathcal{P}_4^m)$, namely that there exists $C>0$ such that 
\begin{equation}\label{P4appo}
	p_{n}R_{m,p_{n}}(x)u_{p_{n}}^{p_{n}-1}(x)\leq C \quad\forall x\in\overline\Omega
	\end{equation}
where $R_{m,p_{n}}(x):=\min_{j=1,\ldots, m}|x-{y}_{j,{n}}|$.
 Arguing by
contradiction we suppose that $$\sup_{x\in\overline\Omega}\left(p_{n} R_{m,p_{n}}(x)
u_{p_{n}}^{p_{n}-1}(x)\right)\to+\infty\quad\textrm{as }n\to+\infty
$$
and let ${{y}}_{m+1,n}\in\overline\Omega$ be such that 
\bel\label{yk+1} p_{n} R_{m,p_{n}}({{y}}_{m+1,n})
u_{p_{n}}^{p_{n}-1}({{y}}_{m+1,n})=\sup_{x\in\overline{\Omega}}\left(p_{n}
R_{m,p_{n}}(x) u_{p_{n}}^{p_{n}-1}(x)\right). \eel By \eqref{yk+1} and since
$\Omega$ is bounded it is clear that
$$
p_{n}u_{p_{n}}^{p_{n}-1}({{y}}_{m+1,n})\to+\infty\quad\textrm{as
$n\to+\infty.$}
$$
Taking the sequences of local maxima ${{y}}_{j,n} \hbox{ for } j = 1,\ldots ,m$  and the added sequence ${{y}}_{m+1,n}$, similarly as in the proof of Proposition \ref{thm:x1N}, we then get that
$(\mathcal{P}_1^{m+1})$, $(\mathcal{P}_2^{m+1})$ and
$(\mathcal{P}_3^{m+1})$ hold.\\
Applying now Lemma \ref{lemma:BoundEnergiaBassino} for the families
of points $({{y}}_{i,n})_{i=1,\ldots,m+1}$ and using \eqref{mSec} we obtain
\[
p_{n}\int_\Omega \left(|\nabla u_{p_{n}}|^2+u_{p_{n}}^2\right)\,dx\geq2\pi\sum_{i=1}^m
{m}_i^2+ 2\pi{m}_{m+1}^2+o_n(1)\overset{\eqref{maxLocMag1}}{\geq}2\pi\sum_{i=1}^m
{m}_i^2+ 2\pi+o_n(1)\
\mbox{ as }n\rightarrow \infty,
\]
thus
$$\displaystyle \lim_{n\rightarrow\infty}
p_{n}\int_\Omega |\nabla u_{p_{n}}|^2+u_{p_{n}}^2\,dx > 2\pi\sum_{i=1}^m
m_i^2 $$ which contradicts \eqref{energiaSemiQuantizzata} concluding the proof of $(\mathcal{P}_4^m)$.

At last in order to derive \eqref{N=k}, let us consider $k$ families of points $x_{1,p_n},x_{2,p_n},\ldots x_{k,p_n}\in\overline\Omega$ as in the statement of Proposition \ref{thm:x1N}. By virtue of Proposition \ref{consequence}
\[
\mathcal S=\{\bar x_1,\ldots,\bar x_m\}=\{\lim_{n\to+\infty}x_{i,p_n}\,:\,i\in\{1,\ldots,k\}\}.
\]
Given $i\in\{1,\ldots,k\}$, let $j\in\{1,\ldots,m\}$ be such that $\lim_{n\to+\infty}x_{i,p_n}=\bar x_j$. Next, recalling that $\{{{y}}_{1,n},{{y}}_{2,n},\ldots,{{y}}_{m,n}\}$ satisfy $(\mathcal{P}_4^m)$ and applying \eqref{P4appo} at $x_{i,p_n}$ we get
\[
p|x_{i,p_n}-{{y}}_{j,n}|u_{p_n}^{p_n-1}(x_{i,p_n})\overset{\eqref{rPiccoloAbbastanza}+\eqref{4.2}}{=}p R_{m,p_n}(x_{i,p_n})u_{p_n}^{p_n-1}(x_{i,p_n})\leq C.
\] 
So in particular, up to a subsequence
\[
\left|\frac{{{y}}_{j,n}-x_{i,p_n}}{\mu_{i,p_n}}\right|\leq C.
\]
As a consequence, up to a subsequence, since ${y}_{j,n}\coorddue=0$ and $x_{i,p_n}$ satisfies $(\mathcal{P}_2^k)$,
there exists $\hat t_{i,j}\in\partial {\R^2_+}$ such that
\[
t_{i,j,n}:=\frac{{{y}}_{j,n}-x_{i,p_n}}{\mu_{i,p_n}}\to \hat t_{i,j}.
\]
By \eqref{vipdef} and \eqref{vip}
\[
0\leq \frac{p_n}{u_{p_n}(x_{i,p_n})}(u_{p_n}({{y}}_{j,n})-u_{p_n}(x_{i,p_n}))=z_{i,p_n}\left( t_{i,j,n}\right)\to U(\hat t_{i,j})\leq0.
\]
Thus, by \eqref{v0} $\hat t_{i,j}=0$, then 
\begin{equation}\label{key}
\frac{|{{y}}_{j,n}-x_{i,p_n}|}{\mu_{i,p_n}}=o_n(1).	
\end{equation}
In conclusion, let us suppose by contradiction that $k>m$, then there exists
\[
i,\ell\in\{1,\ldots,k\},\quad i\neq\ell,\quad \text{such that}\quad \lim_{n\to+\infty}x_{i,p_n}=\lim_{n\to+\infty}x_{\ell,p_n}=\bar x_j\quad\text{for some $j\in\{1,\ldots,m\}$}.
\]
In addition w.l.o.g. let us assume that up to a subsequence $\mu_{i,p_n}\geq\mu_{\ell,p_n}$.\\
By \eqref{key}
\begin{eqnarray*}
\frac{|x_{i,p_n}-x_{\ell,p_n}|}{\mu_{i,p_n}}&\leq&\frac{|x_{i,p_n}-{{y}}_{j,n}|}{\mu_{i,p_n}}+\frac{|x_{\ell,p_n}-{{y}}_{j,n}|}{\mu_{i,p_n}}\\
&\leq&\frac{|x_{i,p_n}-{{y}}_{j,n}|}{\mu_{i,p_n}}+\frac{|x_{\ell,p_n}-{{y}}_{j,n}|}{\mu_{\ell,p_n}}=o_n(1),
\end{eqnarray*}
which is a contradiction against property $(\mathcal{P}_1^k)$ for $x_{1,p_n},x_{2,p_n},\ldots x_{k,p_n}$.
\end{proof}

Next we give a decay estimate for the rescaled functions $w_{j,n}$ which will be fundamental to compute the constants $m_i$'s.

\begin{lemma}\label{lem44}
For any $\gamma
\in(0,2)$ there exists $R_\gamma > 1$ and $n_{\gamma}\in\mathbb N$ such that
\begin{equation}\label{(2.22)}
w_{j,n}(z)\leq (2-\gamma)\log \frac{1}{|z|}+\widetilde{C}_\gamma,
\quad \forall j=1,\ldots,k
\end{equation}
for some $\widetilde{C}_\gamma>0$ provided $R_\gamma \leq |z|\leq
\frac{r}{\varepsilon_{j,n}} ,\ z\in
D_\frac{r}{\varepsilon_{j,n}}(0)$ and $n\geq n_{\gamma}$.
\\
As a consequence
\begin{equation}\label{2.23}
0\leq \left(1+\frac{w_{j,n}(z)}{p_{n}}\right)^{p_{n}}\leq
\left\{\begin{array}{lr}1\qquad \qquad \mbox{ for }|z|\leq R_\gamma\\
C_\gamma\frac{1}{|z|^{2-\gamma}}\quad \mbox{ for }R_\gamma \leq
|z|\leq \frac{r}{\varepsilon_{j,n}}.
\end{array}\right.
\end{equation}
\end{lemma}
\begin{proof} 
Arguing similarly as in  the proof of \cite[Lemma 4.4]{DeMarchisIanniPacellaPositivesolutions}
one can deduce a crucial  pointwise estimate for $w_{j,n}$, namely it can be proved that for any $\varepsilon > 0$, there exist $R_\varepsilon > 1$ and
$n_\varepsilon \in\mathbb N$ such that
\[
\displaystyle w_{j,n}(y)\leq
\big(\frac{\beta_{j,n}}{\pi}-\varepsilon
\big)\log\frac{1}{|y|}+C_\varepsilon, \quad\forall j=1,\ldots, m
\]
for some $C_\varepsilon>0$, provided $2R_\varepsilon \leq |y|\leq
\frac{r}{\varepsilon_{j,n}} ,\ y\in
D_\frac{r}{\varepsilon_{j,n}}(0)$ and $n\geq n_\varepsilon $.\\
\eqref{(2.22)} then follows by  
Lemma \ref{propbeta}. Finally \eqref{2.23} is a direct consequence of \eqref{(2.22)} (see for instance the proof of \cite[Lemma 2.1]{DGIP1} which can be easily adapted to this case). 
\end{proof}

\begin{proposition}\label{valoremi}
 $$m_i=\sqrt{e}, \quad \forall i=1,\ldots,m.$$
\end{proposition}
\begin{proof}
From \eqref{boundEnergiap}
$$c\leq p_{n}\int_{\partial\Omega} u^{p_{n}}_{p_{n}}(x) d\sigma(x)
\leq C$$
hence, by the properties of the Green function $G$,
\begin{eqnarray}\label{(2.25)}
\int_{\partial\Omega\setminus D_{2r}(\bar x_j)}G({{y}}_{j,n},x)u^{p_{n}}_{p_{n}}(x)
\,d\sigma(x)&\leq& C_r\int_{\partial\Omega\setminus
D_{2r}(\bar x_j)}u^{p_{n}}_{p_{n}}(x) \,d\sigma(x)\nonumber\\
&\leq&
  C_r\int_{\partial\Omega}u^{p_{n}}_{p_{n}}(x) \,d\sigma(x) = O(\frac{1}{p_{n}})
\end{eqnarray}
and similarly, observing that \eqref{(2.17)} implies that for $n$ large enough the points
${{y}}_{j,n}\in D_{r/2}(\bar x_j)$  and that $  D_{r/2}(\bar x_j)\subset
D_r({{y}}_{j,n})\subset D_{2r}(\bar x_j)$, also
\begin{eqnarray}\label{(2.26)}
\int_{D_{2r}(\bar x_j)\setminus {D_{ r}({{y}}_{j,n})}
}G({{y}}_{j,n},x)u^{p_{n}}_{p_{n}}(x)\,d\sigma(x) &\leq&
\int_{\{x\in D_{2r}(\bar x_j), \ \frac{r}{2}<|x-\bar x_j|<2r\}}G({{y}}_{j,n},x)u^{p_{n}}_{p_{n}}(x)\,d\sigma(x)\nonumber\\
&\leq&
  C_\frac{r}{2}\int_{\partial\Omega}u^{p_{n}}_{p_{n}}(x) \,d\sigma(x) = O(\frac{1}{p_{n}}). \end{eqnarray}
Using the previous estimates and the Green representation formula, we then get
\begin{eqnarray}\label{(2.27)}
u_{p_{n}}({{y}}_{j,n})&=& \int_{\partial\Omega}G({{y}}_{j,n},x)u^{p_{n}}_{p_{n}}(x)\,d\sigma(x)\nonumber\\
&=&
 \int_{D_{2r}(\bar x_j)}G({{y}}_{j,n},x)u^{p_{n}}_{p_{n}}(x)\,d\sigma(x)+
 \int_{\partial\Omega\setminus D_{2r}(\bar x_j)}G({{y}}_{j,n},x)u^{p_{n}}_{p_{n}}(x)\,d\sigma(x)\nonumber\\
&\overset{\eqref{(2.25)}-\eqref{(2.26)}}{=}&\int_{D_{r}({{y}}_{j,n})}G({{y}}_{j,n},x)u^{p_{n}}_{p_{n}}(x)\,d\sigma(x)
+ o_n(1)\nonumber\\
&\overset{\eqref{defRiscalataMax}}{=}&\frac{u_{p_{n}}({{y}}_{j,n})}{p_{n}}\int_{D_{\frac{r}{\varepsilon_{j,n}}}(0)}G({{y}}_{j,n},{{y}}_{j,n}+\varepsilon_{j,n}z)
\left(1+\frac{w_{j,n}(z)}{p_{n}}\right)^{p_{n}}\,d\sigma(z) + o_n(1)
\nonumber\\
&\overset{\eqref{regularpart}}{=}&\frac{u_{p_{n}}({{y}}_{j,n})}{p_{n}}\int_{D_{\frac{r}{\varepsilon_{j,n}}}(0)}H({{y}}_{j,n},{{y}}_{j,n}+\varepsilon_{j,n}z)
\left(1+\frac{w_{j,n}(z)}{p_{n}}\right)^{p_{n}}\,d\sigma(z)\nonumber\\
&&-\frac{u_{p_{n}}({{y}}_{j,n})}{\pi
p_{n}}\int_{D_{\frac{r}{\varepsilon_{j,p}}}(0)}\log|z|
\left(1+\frac{w_{j,n}(z)}{p_{n}}\right)^{p_{n}}\,d\sigma(z)\nonumber\\
&&-\frac{u_{p_{n}}({{y}}_{j,n})\log(\varepsilon_{j,n})}{\pi
p_{n}}\int_{D_{\frac{r}{\varepsilon_{j,n}}}(0)}
\left(1+\frac{w_{j,n}(z)}{p_{n}}\right)^{p_{n}}\,d\sigma(z) + o_n(1)\nonumber\\
&=&A_n+B_n+C_n+o_n(1).
\end{eqnarray}
Since $H$ satisfies \eqref{Hregularity} in the Appendix, by \eqref{(2.19)} and
\eqref{(2.17)} we get
$$\displaystyle \lim_{n\rightarrow\infty}H({{y}}_{j,n},{{y}}_{j,n}+\varepsilon_{j,n}z)=
H(\bar x_j,\bar x_j), \ \hbox{ for any }z\in\partial\mathbb{R}^2_+,$$
so by \eqref{mSec}, the convergence \eqref{(2.21)} and the uniform
bounds in \eqref{2.23} we can apply the dominated convergence
theorem, and since the function $z\mapsto1 /|z|^{2-\gamma}$ is
integrable in $\{z\in\partial \mathbb{R}^2_+, \ |z|
> R_\gamma\}$ choosing
$\gamma \in (0, 1)$ we deduce
\begin{eqnarray*}
&\displaystyle\lim_{p\rightarrow+\infty}&u_{p_{n}}({{y}}_{j,n})\int_{D_{\frac{r}{\varepsilon_{j,n}}}(0)}H({{y}}_{j,n},{{y}}_{j,n}+\varepsilon_{j,n}z)
\left(1+\frac{w_{j,n}(z)}{p_{n}}\right)^{p_{n}}\,d\sigma(z)\\
&&\overset{\eqref{(2.21)}}{=}m_jH(\bar x_j,\bar x_j)\int_{\partial \R^2_+}e^{U(z)}\,d\sigma(z)
\overset{\eqref{Liouvilleproblem}}{=}2\pi
m_jH(\bar x_j,\bar x_j),
\end{eqnarray*}
from which
\begin{equation}\label{(2.28)}
A_n:=\frac{u_{p_{n}}({{y}}_{j,n})}{p_{n}}\int_{D_{\frac{r}{\varepsilon_{j,n}}}(0)}H({{y}}_{j,n},{{y}}_{j,n}+\varepsilon_{j,n}z)
\left(1+\frac{w_{j,n}(z)}{p_{n}}\right)^{p_{n}}\,d\sigma(z)=o_n(1).
\end{equation}
For the second term in \eqref{(2.27)} we apply again the dominated
convergence theorem, using \eqref{2.23} and observing now that the
function $z\mapsto\log|z| /|z|^{2-\gamma}$ is integrable in
$\{z\in\partial \mathbb{R}^2_+, \ |z|
> R_\gamma\}$  and that $z\mapsto \log |z|$ is integrable in $\{z\in\partial \mathbb{R}^2_+, \ |z|
\leq R_\gamma\}$. Hence we get
$$\displaystyle\lim_{n\rightarrow \infty}u_{p_{n}}({{y}}_{j,n})\int_{D_{\frac{r}{\varepsilon_{j,n}}}(0)}\log|z|
\left(1+\frac{w_{j,n}(z)}{p_{n}}\right)^{p_{n}}\,d\sigma(z)=m_j\int_{\partial
\R^2_+}\log|z|e^{U(z)}\,d\sigma(z)<+\infty$$
and this implies that
\begin{equation}\label{(2.29)}
B_n:=-\frac{u_{p_{n}}({{y}}_{j,n})}{\pi
p_{n}}\int_{D_{\frac{r}{\varepsilon_{j,n}}}(0)}\log|z|
\left(1+\frac{w_{j,n}(z)}{p_{n}}\right)^{p_{n}}\,d\sigma(z)=o_n(1).
\end{equation}
Finally for the last term in \eqref{(2.27)} let us observe that by
the definition of $\varepsilon_{j,n}$ in \eqref{(2.19)}
\begin{equation}\label{(2.30)}
\log\varepsilon_{j,n}=-(p_{n}-1)\log u_{p_{n}}({{y}}_{j,n})-\log p_{n},
\end{equation}
again by the dominated convergence theorem it follows
\begin{eqnarray}\label{(2.32)}
C_n&:=&-\frac{u_{p_{n}}({{y}}_{j,n})\log(\varepsilon_{j,n})}{\pi
p_{n}}\int_{D_{\frac{r}{\varepsilon_{j,n}}}(0)}
\left(1+\frac{w_{j,n}(z)}{p_{n}}\right)^{p_{n}}\,d\sigma(z)\nonumber\\
&=&-\frac{u_{p_{n}}({{y}}_{j,n})\log(\varepsilon_{j,n})}{\pi p_{n}}\left(\int_{\partial
\R^2_+}e^{U(z)}\,d\sigma(z)+o_n(1)\right)\nonumber\\
&=&-\frac{u_{p_{n}}({{y}}_{j,n})\log(\varepsilon_{j,n})}{\pi p_{n}}(2\pi+o_n(1))\nonumber\\
&\overset{\eqref{(2.30)}}{=}&u_{p_{n}}({{y}}_{j,n})\left[\frac{p_{n}-1}{p_{n}}\log
u_{p_{n}}({{y}}_{j,n})+\frac{\log p_{n}}{p_{n}}\right](2+o_n(1)).
\end{eqnarray}
Substituting \eqref{(2.28)}, \eqref{(2.29)} and \eqref{(2.32)} into
\eqref{(2.27)} we get
$$u_{p_{n}}({{y}}_{j,n})=u_{p_{n}}({{y}}_{j,n})\left[\frac{p_{n}-1}{p_{n}}\log
u_{p_{n}}({{y}}_{j,n})+\frac{\log p_{n}}{p_{n}}\right](2+o_n(1))+o_n(1),$$
 passing
to the limit as $n\rightarrow \infty$ and using \eqref{mSec} we
conclude that
$$\log m_j=\frac{1}{2}.$$
\end{proof}

\subsection{The proof of Theorem \ref{teo:Positive}}$\;$\\

The statements of Theorem \ref{teo:Positive} have been proved in the
various propositions obtained so far. In particular \textit{(i)} is a consequence of Lemma \ref{defgammaj}, \eqref{valoreGamma} and Proposition \ref{valoremi}.
\textit{(ii)} derives from 
\eqref{mSec} and Proposition
\ref{valoremi}. The energy limit \textit{(iii)} follows from \eqref{energiaSemiQuantizzata} in Proposition
\ref{proposizione:quasiQuanteConclusione}, combined with Proposition \ref{valoremi}. The statement \textit{(iv)} is contained in Lemma \ref{lemma:possoRiscalareAttornoAiMax} in the flat case, and can be easily extended to the non-flat case, similarly as in \cite{F, Castro1}, see Subsection \ref{sectionchange}.

\

\appendix
\section{Some properties of the Green function}
\label{SectionAppendix}
 Let $y\in\partial \Omega$ and let $G(x, y)$ be the
Green function satisfying the Neumann problem \eqref{Greenequation}.
First note that $G\geq 0 $ and by the classical strong maximum
principle, for each $y\in \partial \Omega$ $G(. ,y)$ cannot attain
its minimum in $\Omega$. Also, by the Hopf lemma if $G(x, y) = 0$
 for some $x,\,y\in \partial\Omega, \ x\neq y$ then the normal derivative $\frac{\partial G}{\partial \nu_x} (x, y)$
 is negative, which is impossible. Therefore, for each $y\in \partial \Omega$ we have
\begin{equation}\label{Gpositive}
G(\cdot , y) > 0 \hbox{ in }\overline{ \Omega}.
\end{equation}
 By a compactness argument we can find a constant $c > 0$ such that
 $G(x, y)>c$ for all $y\in\partial \Omega$  and all $x\in \overline{\Omega}$.
 \begin{lemma}\label{Gponctuelestimate}There exists a positive constant $C_1$ such that
 $$0<G(x,y)\leq C_1\left(|\log|x-y||+1\right)\quad\hbox{for each }x\in \overline{\Omega}\setminus\{y\}
  \hbox{ and }y\in\partial \Omega  .$$
\end{lemma}
\begin{proof}
By \eqref{regularpart}, we have
\begin{equation}\label{Greenfunction}
G(x, y)=  \frac{1}{\pi}\log \frac{1}{ |x - y|} + H(x, y)
\end{equation}
where $\frac{1}{\pi}\log \frac{1}{ |x - y|}$ is the singular part of
$G$ and $H(x, y)$ is the regular part of $G$. The function $H(.,y)$
satisfies
\begin{equation*}\label{regular part eq}
\left\{\begin{array}{lr}\displaystyle-\Delta_x H(x,y)+H(x, y)= -\frac{1}{\pi}\log\frac{1}{|x-y|}\qquad  \mbox{ in }\Omega\\
\displaystyle\frac{\partial H}{\partial \nu_x}(x,y)=
\frac{1}{\pi}\frac{\langle x-y, \nu(x)\rangle}{|x-y|^2}\qquad\mbox{ on }\partial
\Omega.
\end{array}\right.
\end{equation*}
Arguing as in \cite{WW} (see pages 834 and 835), we have
\begin{equation}\label{Hregularity}
x\mapsto H(x,y)\in C^{1,\gamma}(\overline{\Omega}), \ y\mapsto
H(x,y)\in
C^{1,\gamma}(\partial\Omega,C^{1,\gamma}(\overline{\Omega}))\hbox{
and }\nabla_xH\in C(\overline{\Omega}\times \partial \Omega)
\end{equation}
for any $\gamma \in (0,1)$. The desired result follows from
\eqref{Gpositive}, \eqref{Greenfunction} and
\eqref{Hregularity}.
\end{proof}
As consequence of Lemma \ref{Gponctuelestimate}, we have the
following result.
\begin{lemma}\label{gestimate}There
exist $C_2,C_\delta > 0$ such that:
\begin{equation}\label{Gboundness}
G(x, y) \leq C_\delta \quad \forall \ |x - y| >\delta > 0,
\end{equation}
\begin{equation}\label{gradGboundness}
|\nabla_x G(x, y) |\leq\frac{ C_2}{|x-y|} \quad \forall \
x\in\overline{\Omega}\setminus \{y\}.
\end{equation}
\end{lemma}
\begin{proof}
It is easy to see that \eqref{Gboundness} is a consequence of Lemma
\ref{Gponctuelestimate}.\\
By \eqref{Greenfunction} we have
\begin{equation}\label{gradgreen}
\nabla_xG(x, y)= - \frac{1}{\pi}\frac{x-y}{ |x - y|^2} +
\nabla_xH(x, y)
\end{equation}
for each $x\in\overline{\Omega}\setminus \{y\}$. Hence
\eqref{gradGboundness} follows from \eqref{gradgreen} and
\eqref{Hregularity}.
\end{proof}
Let $x_1,\ldots,x_n$ be $n$ distinct points in $\partial \Omega$ and
let $r$ be some positive small constant such that $B_r(x_i)\cap
B_r(x_j)=$ for all $1\leq i\neq j\leq n$.

\begin{lemma}\label{Greenestimate}Let $1\leq i\leq n$ and let $(c_j)_{1\leq j\leq n}$ be $n$ real numbers.
For each $x\in \overline{B_{r}(x_i)\cap
 \Omega}\setminus\{x_i\}$, we have
$$
\sum_{j=1}^nc_j G(x,x_j)=\frac{c_i}{\pi}\log \frac{1}{ |x
- x_i|} + O(1)\quad\hbox{and}\quad\sum_{j=1}^nc_j\nabla
G(x,x_j)=- \frac{c_i}{\pi}\frac{x-x_i}{|x-x_i|^2} + O(1) .$$
\end{lemma}
\begin{proof}
Using Lemma \ref{gestimate}, for each $x\in \overline{B_{r}(x_i)\cap
 \Omega}\setminus\{x_i\}$ we have
\begin{eqnarray}
\sum_{j=1}^mc_j G(x,x_j)=c_iG(x,x_i) + O(1)
\quad\hbox{and}\quad \sum_{j=1}^mc_j\nabla G(x,x_j)=c_i
\nabla G(x,x_i) + O(1).\nonumber
\end{eqnarray}
Furthermore $G(x,x_i)$ satisfies \eqref{Greenfunction} and
\eqref{gradgreen}, so that, by the regularity of $H$ in
\eqref{Hregularity} we obtain the desired result.
\end{proof}

{\bf{Acknowledgment:}}\\
Some parts of this work were done when the second author was hosted
by university of Rome La Sapienza and Rutgers University of
New-Jersey when he won the inaugural Abbas Bahri Excellence
Fellowship. He wishes to thank these institutions for support and
good working conditions. We also thank H. Castro for useful clarifications about his results.

\end{document}